\documentclass{amsart}
\usepackage{amsmath, amsthm, amsfonts, amssymb}
\usepackage{mathrsfs}
\usepackage{microtype}
\usepackage[utf8]{inputenc}
\providecommand{\noopsort[1]{}}
\usepackage{bbm}
\usepackage{dsfont}
\usepackage{ifthen}
\usepackage{nicefrac}
\numberwithin{equation}{section}
\usepackage{a4}
\usepackage[active]{srcltx}
\usepackage{verbatim}
\usepackage{hyperref}
\usepackage{xcolor}

\usepackage[shortlabels]{enumitem}
\setlist{leftmargin=*}
\setlist[1]{labelindent=1.2\parindent}

\newtheorem{thm}{Theorem}[section]
\newtheorem*{thm*}{Theorem}
\newtheorem{cor}[thm]{Corollary}
\newtheorem{prop}[thm]{Proposition} 
\newtheorem{lem}[thm]{Lemma}

\theoremstyle{remark}
\newtheorem{rem}[thm]{Remark}

\newtheorem{example}[thm]{Example}

\theoremstyle{definition}
\newtheorem{defn}[thm]{Definition}
\newtheorem{hyp}[thm]{Hypothesis}        
\newtheorem*{defn*}{Definition}
\newcommand{\coloneqq}{\mathrel{\mathop:}=}

\renewcommand{\Re}{{\rm Re}\,}

\newcommand{\eps}{\varepsilon}
\newcommand{\tr}{\operatorname{tr}}

\newcommand{\supp}{\operatorname{supp}}
\newcommand{\one}{\mathds{1}}

\newcommand{\CR}{\mathds{R}}
\newcommand{\CC}{\mathds{C}}

\newcommand{\CN}{\mathds{N}}

\newcommand{\cM}{\mathscr{M}}

\newcommand{\cL}{\mathscr{L}}
\newcommand{\cA}{\mathscr{A}}

\newcommand{\la}{\langle}
\newcommand{\ra}{\rangle}

\newcommand{\dx}{\mathrm{d}}
\global\long\def\diff#1#2{\,#1\mathrm{d}#2}
\global\long\def\tr{\operatorname{tr}}

\global\long\def\dual#1{\langle#1\rangle}

\begin{document}
\title{Diffusion with nonlocal Robin boundary conditions}
\author{Wolfgang Arendt}
\email{wolfgang.arendt@uni-ulm.de}
\address{Institute of Applied Analysis, Ulm University, 89069 Ulm, Germany}

\author{Stefan Kunkel}
\email{stefan.kunkel@uni-ulm.de}
\address{Graduiertenkolleg 1100, Ulm University, 89069 Ulm, Germany}
\author{Markus Kunze}
\email{markus.kunze@uni-konstanz.de}
\address{Universit\"at Konstanz, Fachbereich Mathematik und Statistik, 78467 Konstanz, Germany}

\begin{abstract}
We investigate a second order elliptic differential operator $A_{\beta, \mu}$
on a bounded, open set $\Omega\subset\CR^{d}$ with Lipschitz boundary
subject to a nonlocal boundary condition of Robin type. More precisely
we have $0\leq \beta\in L^{\infty}(\partial\Omega)$ and $\mu\colon\partial\Omega\to\cM(\overline{\Omega})$,
and boundary conditions of the form
\[
\partial_{\nu}^{\cA}u(z)+\beta(z)u(z)=\int_{\overline{\Omega}}u(x)\diff{\mu(z)(}{x)},\ z\in\partial\Omega,
\]
where $\partial_{\nu}^{\cA}$ denotes the weak conormal derivative with respect to our differential operator.
Under suitable conditions on the coefficients of the differential
operator and the function $\mu$ we show that $A_{\beta, \mu}$
generates a holomorphic semigroup $T_{\beta,\mu}$ on $L^{\infty}(\Omega)$ which enjoys the strong Feller property. In particular, it  
takes values in $C(\overline{\Omega})$. Its restriction to $C(\overline{\Omega})$ is strongly continuous and holomorphic.
We also establish positivity and contractivity of the semigroup under additional assumptions and study the
asymptotic behavior of the semigroup.
\end{abstract}

\subjclass[2010]{47D07, 60J35, 35B35}
\keywords{Diffusion process, non-local boundary condition, stability}

\maketitle

\section{Introduction}

In the 1950s Feller \cite{feller-semigroup, feller-diffusion, feller57} described all diffusion processes in one dimension; in particular, he characterized the boundary conditions which lead to generators of what today is called a \emph{Feller semigroup}. Besides the classical Dirichlet, Neumann and Robin boundary conditions, also certain nonlocal boundary conditions can occur. In higher dimensions, it was Ventsel' \cite{wenzell} who first described the boundary conditions satisfied by the functions in the domain of the generator of a Feller semigroup. Naturally, the converse question of which of these boundary conditions actually lead to generators of Feller semigroups
has recieved a lot of attention. The starting point for that question is the article by Sato and Ueno \cite{su65}, who poved
that this is the case if and only if a certain auxiliary problem (which is a generalization of the Dirichlet problem, involving the 
boundary condition in question; cf.\ Equation \eqref{eq.poisson} below) is solvable for sufficiently many right-hand sides. Some concrete examples of boundary
conditions for which one obtains a generator of a Feller semigroup were already contained in \cite{su65, wenzell}; more refined results 
were obtained by Taira, see \cite{taira} and the references therein, Skubachevski{\u\i} \cite{sk89, sk95} and Galakhov and
Skubachevski{\u\i} \cite{galaskub}.\medskip 

In this article, we are concerned with diffusion equations with certain non-local boundary conditions of Robin type. 
Let us describe this in 
more detail. We consider a bounded, open set $\Omega\subset\CR^d$ with Lipschitz boundary. As far as our boundary condition is concerned, we make the following assumptions.

\begin{hyp}\label{h.mu}
We are given a real-valued function $0 \leq \beta \in L^\infty(\partial\Omega)$, where $\partial\Omega$ is endowed with
surface measure $\sigma$. Moreover, we are given a map $\mu : \partial \Omega \to \mathscr{M}(\overline{\Omega})$, the space of complex-valued measures on $\overline{\Omega}$, which satisfies the following  conditions.
\begin{enumerate}
[(a)]
\item For every function $f \in B_b(\overline{\Omega})$, the space of all bounded and Borel measurable functions on 
$\overline{\Omega}$, the map $z\mapsto \langle f, \mu (z)\rangle \coloneqq \int_{\overline{\Omega}} f(x)\mu(z)(\mathrm{d}x)$ is 
measurable;
\item for some $p>d-1$ with $p\geq 2$ we have $\int_{\partial{\Omega}} \|\mu (z)\|^p\, \dx\sigma (z)<\infty$ and
\item there exists a positive and bounded measure $\tau$ on $\overline{\Omega}$ such that for every $z\in \partial \Omega$
the measure $\mu (z)$ is absolutely continuous with respect to $\tau$.
\end{enumerate}
\end{hyp}

In (a), it actually suffices to assume that the map $z\mapsto \langle f, \mu (z)\rangle$ is measurable for all $f \in C(\overline{\Omega})$. The measurability for those $f$ which are merely bounded and measurable follows by a monotone class argument, cf.\ the proof of
Lemma 6.1 in \cite{k11}.
We will see later on that if instead of (a) we assume
\begin{enumerate}[(a$^\prime$)]
\item For every $f \in B_b(\overline{\Omega})$ the map $z \mapsto \langle f, \mu (z)\rangle$ is continuous
\end{enumerate}
then parts (b) and (c) in Hypothesis \ref{h.mu} are automatically satisfied.

Assuming Hypothesis \ref{h.mu} we can define the operator $\Delta_{\beta, \mu}$ on $L^\infty(\Omega)$ by
\begin{align*}
	D(\Delta_{\beta, \mu}) &:=\{u\in H^1(\Omega)\cap C(\overline{\Omega}):\ \Delta u\in L^\infty(\Omega),\\
	&\qquad\qquad\partial_\nu u(z)+\beta(z) u(z)=\dual{u, \mu(z)}\,\, \forall\, z\in \partial \Omega\}\\
	\Delta_{\beta, \mu} u&=\Delta u.
\end{align*}
Here $H^1(\Omega)$ is the usual Sobolev space and the \emph{normal derivative} $\partial_\nu u$ has to be understood as follows.

\begin{defn}
	For a function $u\in H^1 (\Omega)$, we write $\tr u$ for its trace in $L^2(\partial \Omega)$. Let $u \in H^1(\Omega)$ be
	such that $\Delta u\in L^2(\Omega)$ and let $h\in L^2(\partial\Omega)$. We say that $\partial_\nu u=h$ if Green's formula
	\[
		\int_{\Omega}\Delta uv \dx x+\int_{\Omega}\nabla u \nabla v \dx x= \int_{\partial\Omega} h\tr v \dx \sigma
	\]
	holds for all $v\in H^1(\Omega)$.
\end{defn}

In what follows we will not distinguish between a function $u\in H^1(\Omega)$ and its
trace $\tr u$ in integrals over the boundary $\partial \Omega$.

With this definition of the normal derivative the operator $\Delta_{\beta, \mu}$ is well-defined.  Indeed, if $u\in D(\Delta_{\beta, \mu})$
then $u\in C(\overline{\Omega})$ whence
\[
	h(z):=\dual{u,\mu(z)}-\beta(z)u(z)
\]
defines a function $h\in  L^2(\partial\Omega)$. Since furthermore
$u\in H^1(\Omega)$ and $\Delta u \in L^\infty(\Omega)\subset L^2(\Omega)$  it makes sense to say that $\partial_\nu u= h$. This condition is the \emph{Robin boundary condition} we are interested in with \emph{local} part $\beta \tr u$ and  \emph{non-local} part $\dual{u, \mu(\cdot)}$.

We also consider the part $\Delta_{\beta, \mu}^C$ of $\Delta_{\beta, \mu}$ in $C(\overline{\Omega})$ given by
\begin{align*}
	D(\Delta_{\beta, \mu}^C) &:=\{u\in H^1(\Omega)\cap C(\overline{\Omega}):\ \Delta u\in C(\overline\Omega),\\
	&\qquad\qquad\partial_\nu u+\beta u|_{\partial\Omega}=\dual{u,\mu(\cdot)}\}\\
	\Delta_{\beta, \mu}^C u=\Delta u.
\end{align*}
One of our main results is the  following generation theorem.

\begin{thm}\label{t.laplace}
Assuming Hypothesis \ref{h.mu}, the operator $\Delta_{\beta, \mu}$ generates a holomorphic semigroup $(T_{\beta,\mu}(t))_{t>0}$ on $L^\infty(\Omega)$ which satisfies the strong Feller property. In particular, this semigroup leaves the space $C(\overline{\Omega})$ invariant. Its restriction to $C(\overline{\Omega})$  is a strongly continuous and holomorphic semigroup whose generator is 
$\Delta_{\beta, \mu}^C$.
\end{thm}

We refer to Section \ref{s.prelim} for the definition of holomorphic semigroups which are not strongly continuous at $0$ and for an explanation of the strong Feller property. We will actually prove Theorem \ref{t.laplace} in more generality, replacing the Laplacian with a general second order strictly elliptic differential operator with measurable coefficients.

We will also establish positivity and contractivity of the semigroup $T_{\beta, \mu}$ under additional assumptions on 
$\beta$ and $\mu$, see Section \ref{s.non}. In the case of Theorem \ref{t.laplace}, where we consider the Laplacian, the conditions are as follows. 
If the measures $\mu(z)$ are positive for all $z\in \partial \Omega$ then the semigroup $T_{\beta,\mu}$ is \emph{positive}; i.e. each $T_{\beta,\mu}(t)$ leaves the positive cone $L^\infty(\Omega)_+$ of $L^\infty(\Omega)$ invariant. If additionally we have that
\begin{equation}
\label{eq.submarkov}
\mu (z, \overline{\Omega}) \leq \beta (z) \quad \mbox{for almost all } z \in \partial \Omega,
\end{equation}
then the semigroup $T_{\beta, \mu}$ is \emph{sub-Markovian}, i.e.\ $T_{\beta,\mu}$ is positive
and $T_{\beta,\mu}(t)\one \leq \one$ for all $t>0$. If equality holds in \eqref{eq.submarkov} then $T_{\beta,\mu}$ is \emph{Markovian}, i.e. $T_{\beta, \mu} (t)$ is positive and $T_{\beta, \mu}(t)\one = \one$. In these situations we will also study the asymptotic behavior of the semigroup $T_{\beta, \mu}$. In the sub-Markovian, non-Markovian case the semigroup converges in operator norm to 0, whereas in the Markovian case the orbits converge to an equilibrium.
\medskip

Let us compareour results to the existing literature. First of all, in this article 
we consider less restrictive assumptions on the coefficients and the domain. 
Indeed, in the above mentioned references, the domain and the coefficients of
the operator are assumed to be smooth (i.e.\ $C^\infty$ or a suitable H\"older continuity), whereas here we consider coefficients which are merely measurable and a domain with Lipschitz boundary. Moreover, we prove our generation result for general boundary conditions and study additional properties, such as positivity and the Markov property, afterwards, whereas in  \cite{sk89, sk95, galaskub, taira} there are a priori assumptions imposed on the coefficients in the boundary condition which ensure these properties. On the other hand, the quoted result treat more general boundary conditions which cover also, e.g., viscosity phenomena on the boundary.

Possibly the most important novelty in this article is that we obtain a holomorphic semigroup on $C(\overline{\Omega})$, even
on $L^\infty(\Omega)$. So far, holomorphic semigroups for diffusion processes with nonlocal boundary conditions were only
established on the $L^p$-scale ($1\leq p <\infty$), see \cite{taira_analytic}. 
To the best of our knowledge, the only other article which establishes 
holomorphy of the semigroup on $C(\overline{\Omega})$ for diffusion operators with non-local boundary conditions is our previous article \cite{akk16}, where we have treated non-local Dirichlet boundary conditions. We should note that the two problems are rather different. Indeed, the non-local Robin boundary condition 
considered here falls in the so-called `transversal case', where, due to the normal derivative, 
the non-local term has lower order than the 
rest of the boundary condition. This is not the case for the non-local Dirichlet boundary condition which falls in the so-called `non-transversal case'. Also the strategy for the proof is rather different. The proof of Theorem \ref{t.laplace} is based on a perturbation result 
by Greiner, which we explain in Section \ref{s.2}. We will actually present a slight generalization of Greiner's result which establishes 
additional properties of the perturbed semigroup. We should mention that Greiner's perturbation result cannot be used in the case
of non-local Dirichlet boundary conditions where the maximum principle plays an essential role. 

The holomorphy of the semigroup toghether with the compactness allows us to study the asymptotic behavior of the semigroup in Section 6.

Non-local Robin boundary conditions of the above form occur in several concrete situations, for example in heat control, where the heat is measured in the interior and the control is via the boundary, see \cite{bkl01, gjs08}.

The structure of this article is as follows. After some preliminaries in Section 2, we present Greiner's boundary perturbation, along 
with our modifications, in Section 3.  Section 4 contains results on elliptic differential operators with local Robin boundary conditions which are needed subsequently. In Section 5 we prove our main generation result. Section 6 contains our results concerning the asymptotic behavior of the semigroup and Section 7 is devoted to the special situation where all measures $\mu (z)$ are absolutely continuous with respect to Lebesgue measure. There we will see that our conditions for positivity and sub-Markovianity are necessary in this situation. The concluding Section 8  contains some examples where Hypothesis \ref{h.mu} is satisfied, in particular, we prove that it is satisfied whenever condition (a$^\prime$) is fulfilled. In the appendix we present some general results on the asymptotic behavior of positive semigroups, which we use in Section 6.

\section{Preliminaries}\label{s.prelim}

\subsection{Semigroups that are not necessarily strongly continuous} In this article, we shall consider semigroups 
on the space $L^\infty(\Omega)$, where $\Omega$ is a bounded open subset of $\CR^d$. By a result of Lotz \cite{lotz} (see 
also \cite[Corollary 4.3.19]{abhn}), a strongly continuous semigroup on $L^\infty(\Omega)$ necessarily has a bounded generator.
As we are concerned with second order differential operators, we will encounter semigroups that are not strongly continuous. 
Since this is not a standard situation, we recall the relevant definitions and results here. Let us start with the following definition, 
taken from \cite[Section 3.2]{abhn}.

\begin{defn}
Let $X$ be a Banach space. A \emph{semigroup} is a strongly continuous mapping $T: (0,\infty) \to \mathscr{L}(X)$
such that
\begin{enumerate}
[(a)]
\item $T(t+s)=T(t)T(s)$ for all $t,s>0$;
\item there exist constants $M>0$ and $\omega \in \CR$ such that $\|T(t)\|\leq Me^{\omega t}$
for all $t>0$;
\item if $T(t)x =0$ for all $t>0$, it follows that $x=0$.
\end{enumerate}
We say that $T$ is of \emph{type} $(M,\omega)$ to emphasize that $(b)$ holds with these constants. A semigroup of type
$(1,0)$ is called \emph{contraction semigroup}. If additionally we have 
\[
T(t)x \to x \quad\mbox{as}\quad t\to 0
\]
for all $x\in X$, then $T$ is called \emph{strongly continuous}.
\end{defn}

Clearly, the condition $T(t)x \to x$ as $t\to 0$ for every $x\in X$ implies condition (c) above and it is not difficult to see
that it also implies condition (b) (see \cite[I Proposition 5.5]{en}). Thus, our definition of strongly continuous semigroup
coincides with the classical definition used, e.g., in \cite[I, Definition 5.1]{en}. However, even without strong continuity, we 
can associate a generator with a semigroup. Indeed, if $T$ is a semigroup of type $(M,\omega)$, then there exists a unique
operator $G$ such that $(\omega, \infty)$ is contained in the resolvent set $\rho (G)$ of $G$ and
\[
R(\lambda, G)x = \int_0^\infty e^{-\lambda t}T(t)x\, dt
\]
for all $x\in X$ and $\lambda > \omega$, see \cite[Equation (3.13)]{abhn}. The operator $G$ is called the \emph{generator} of $T$. 
Note that in the case of strongly continuous semigroups this is equivalent to the usual `differential' definition of the generator, 
see \cite[II Theorem 1.10]{en}.\smallskip

A semigroup $T$ is called \emph{holomorphic}, if there is some angle $\theta \in (0,\frac{\pi}{2}]$ such that $T$ has a holomorphic
extension to the sector 
\[
\Sigma_\theta \coloneqq \big\{ re^{i\varphi} : r >0, |\varphi |<\theta\big\}
\]
which is bounded on $\Sigma_\theta \cap \{z\in \CC: |z|\leq 1\}$, see \cite[Definition 3.7.1]{abhn}.

The generators of holomorphic semigroups can be characterized as follows.

\begin{thm}\label{t.holgen}
An operator $G$ on $X$ generates a holomorphic semigroup if and only if there exists a constant $\omega \in \CR$ sucht 
that $\{\lambda \in \CC : \Re\lambda >\omega\} \subset \rho (G)$ and
\[
\sup_{\Re\lambda >\omega} \|\lambda R(\lambda, G)\| < \infty.
\]
\end{thm}

\begin{proof}
\cite[Proposition 2.1.11]{lunardi} or \cite[Corollary 3.7.12 and Proposition 3.7.4]{abhn}.
\end{proof}

The following Lemma is taken from \cite[Proposition 2.1.4]{lunardi}.

\begin{lem}\label{l.sc}
Let $T$ be a holomorphic semigroup with generator $G$. Then we have $T(t)x \to x$ if and only if
$x \in \overline{D(G)}$.
\end{lem}

It follows from Lemma \ref{l.sc} that a holomorphic semigroup is strongly continuous if and only if its generator is densely defined. Recalling
from \cite[Proposition 2.1.1]{lunardi}
that $D(G)$ (and hence also $\overline{D(G)}$) is invariant under $T$, a second 
corollary of Lemma \ref{l.sc} is that every holomorphic semigroup $T$ restricts to a strongly continuous and holomorphic
semigroup on $\overline{D(G)}$.

\subsection{Transition kernels and the strong Feller property} In the study of Markov processes it is important
that the transition semigroup consists of \emph{kernel operators}, as these give the transition probabilities of the process.
We recall the relevant definitions and results and introduce the \emph{strong Feller property} which is important for the ergodic theory of Markov processes. In this subsection, $K$ is a compact metric space and $\mathscr{B}(K)$ denotes the Borel $\sigma$-algebra on $K$. Later on, we will consider $K= \overline{\Omega}$.

A (bounded) \emph{kernel} on $K$ is a map $k: K\times \mathscr{B}(K) \to \CC$ such that
\begin{enumerate}
[(i)]
\item the map $x\mapsto k(x,A)$ is Borel-measurable for all $A\in \mathscr{B}(K)$,
\item the map $A \mapsto k(x,A)$ is a (complex) measure on $\mathscr{B}(K)$ for each $x\in K$ and
\item we have $\sup_{x\in K} |k|(x,K)<\infty$, where $|k|(x,\cdot)$ denotes the total variation of the measure $k(x,\cdot)$.
\end{enumerate}

Let $X= C(K)$ or $X=B_b(K)$. We call an operator $T\in \mathscr{L}(X)$ a \emph{kernel operator} if there exists a kernel $k$ 
such that
\begin{equation}\label{eq.kernel}
Tf(x) = \int_K f(y)\, k(x,dy)
\end{equation}
for all $f\in X$ and $x\in K$. As there is at most one kernel $k$ satisfying the above equation, we call $k$ the \emph{kernel
associated with $T$}. Conversely $T$ is called the \emph{operator associated with $k$}.

Let us note that every bounded operator on $C(K)$ is a kernel operator, since given $T\in \mathscr{L}(X)$ we can set
$k(x,\cdot) \coloneqq T^*\delta_x \in \mathscr{M}(K)$ for every $x \in K$. Standard arguments (cf.\ \cite[Proposition 3.5]{k11}) show
that $k$ is indeed a kernel and it is then easy to see that $T$ is associated with $k$. On the other hand, not every bounded
operator on $B_b(K)$ is a kernel operator. We have the following characterization.

\begin{lem}\label{l.kernel}
Let $T\in \mathscr{L}(B_b(K))$. The following are equivalent.
\begin{enumerate}
[(i)]
\item $T$ is a kernel operator.
\item $T$ is pointwise continuous, i.e. if $f_n$ is a bounded sequence converging pointwise to $f$, then $Tf_n$ converges pointwise
to $Tf$.
\end{enumerate}
\end{lem}

\begin{proof}
The implication `(i) $\Rightarrow$ (ii)' follows from the dominated convergence theorem. For the converse, put $k(x,A) \coloneqq 
(T\one_A)(x)$, where $\one_A$ denotes the indicator function of the set $A\in \mathscr{B}(K)$. Using (ii), we see that $k(x,\cdot)$ is a measure, thus $k$ is a kernel. By the density of simple functions in $B_b(K)$ with respect to the supremum norm, we easily see that
$T$ is associated with $k$.
\end{proof}

Let us note that given a kernel operator $T$ on $C(K)$, we can always extend $T$ to a kernel operator $\tilde T$ on 
$B_b(K)$ by defining $(\tilde T f)(x)$ by the right-hand side of \eqref{eq.kernel} for $f\in B_b(K)$. The operator $\tilde T$ is called the \emph{canonical extension} of $T$. The operator $T$ may have
other extensions to a bounded operator on $B_b(K)$, but $\tilde T$ is the only one which is a kernel operator.

\begin{defn}
A kernel operator $T$ on $B_b(K)$ is called \emph{strong Feller operator} if $Tf \in C(K)$ for every $f\in B_b(K)$. A kernel operator 
$T$ on $C(K)$ is called \emph{strong Feller operator} if its canonical extension $\tilde T$ is a strong Feller operator.
\end{defn}

Let us now consider a bounded, open set $\Omega \subset \CR^d$ and put $K \coloneqq \overline{\Omega}$. In what follows,
we will be concerned with operators $T\in \mathscr{L}(L^\infty(\Omega))$ which take values $C(K)$. It would be tempting to also call such an operator a strong Feller operator, but there are some subtleties in this situation. Let us explain this a little bit.

The concept of `strong Feller operator' is only useful for kernel operators. Give an operator $T\in \mathscr{L}(L^\infty(\Omega))$
which takes values in $C(K)$, we can consider the restriction $S\coloneqq T|_{C(K)}$ of $T$ to $C(K)$. As observed above,
$S$ is a kernel operator and thus has a canonical extension $\tilde S$ to $\mathscr{L}(B_b(K))$. Now let
$\iota: B_b(K) \to L^\infty(\Omega)$ map a bounded measurable function to its equivalence class modulo equality almost everywhere.
Then the obvious question is whether $T \circ \iota = \tilde S$. Example 5.4 in \cite{akk16} shows that this need not be the case
without further assumptions. The problem is that $T\circ \iota$ need not be a kernel operator.
However, using the characterization of kernel operators in \ref{l.kernel}, we obtain

\begin{lem}\label{l.sf}
Let $T\in \mathscr{L}(L^\infty(\Omega))$ take values in $C(\overline{\Omega})$ and let $\iota: B_b(\overline{\Omega}) \to
L^\infty (\Omega)$ be as above. Then $T\circ\iota$ is a kernel operator if and only if for every bounded sequence $(f_n)\subset L^\infty(\Omega)$ converging almost everywhere to $f$, we have $Tf_n(x) \to Tf(x)$ for all $x\in \overline{\Omega}$. In this case,
$T\circ\iota$ is a strong Feller operator. 
\end{lem}

We define:

\begin{defn}\label{def.sf}
An operator $T\in \mathscr{L}(L^\infty(\Omega))$ is called \emph{strong Feller operator} if
\begin{enumerate}
[(a)]
\item $Tf \in C(\overline{\Omega})$ for every $f\in L^\infty(\Omega)$ and
\item For every bounded sequence $(f_n)\subset L^\infty(\Omega)$ converging pointwise almost everywhere to $f$, we have
$Tf_n \to Tf$ pointwise.
\end{enumerate}
\end{defn}

\section{Greiner's boundary perturbation revisited}\label{s.2}

An important tool in this article is boundary perturbation of the generator of a holomorphic semigroup, established by Greiner in his seminal
article \cite{greiner}. As a matter of fact, we need some extensions of Greiners results whose proofs follow along the lines of 
Greiners article with minor modifications. More precisely, we will consider semigroups which are not necessarily strongly continuous. Besides being interesting in its own right, this will allow us to establish under appropriate assumptions the strong Feller property
for the perturbed semigroup. Likewise, other modifications allow us to prove compactness, positivity and domination for the perturbed semigroup. In an effort of being self contained and for the convenience of the reader we provide complete proofs.

Throughout this section, we make the following assumption.

\begin{hyp}\label{hyp.greiner}
We are given complex Banach spaces $(X, \|\cdot\|_X)$, $(D, \|\cdot\|_D)$ and $(\partial X, \|\cdot\|_{\partial X})$, where $D$ is continuously 
embedded into $X$. We identify $D$ with its image in $X$ and frequently consider the closure $\overline{D}$ of $D$ in $X$. 
Moreover, we are given a continuous \emph{maximal operator} $A : D \to X$, a continuous \emph{boundary operator} $B: D \to \partial X$ and a \emph{boundary 
perturbation} $\Phi : \overline{D} \to \partial X$. We assume that all of these mappings are linear and continuous. Moreover, we assume
the following.
\begin{enumerate}
[(a)]
\item The boundary operator $B$ is surjective;
\item the boundary perturbation $\Phi$ is compact;
\item the operator $A_0 \coloneqq A|_{\ker B}$ generates a holomorphic semigroup on $X$ and we have
$\overline{D(A_0)} = \overline{D}$. We denote by $\omega$ a real number such that any $\lambda \in \CC$ with
$\Re \lambda > \omega$ belongs to $\rho (A_0)$.
\end{enumerate}
\end{hyp}

In comparison to Greiner's original work, the main difference in our assumption is that we do not assume the operator $A_0$
to be densely defined in $X$. Consequently, the semigroup $T$ generated by $A_0$ need not be strongly continuous. However, since
$\overline{D(A_0)} = \overline{D}$, it follows from Lemma \ref{l.sc} that for every $f \in \overline{D}$ the orbit $t \mapsto T(t)f$ is strongly continuous on $[0,\infty)$ and $T$ restricts to a strongly continuous holomorphic semigroup on $\overline{D}$.
\smallskip

Given the above maps, we define the perturbed operator $A_\Phi$ by 
\[
D(A_\Phi) \coloneqq \{ u \in D : Bu=\Phi u \}, \quad A_\Phi u = Au.
\]
We can now formulate our version of Greiner's result.

\begin{thm}\label{t.greiner}
Assuming Hypothesis \ref{hyp.greiner}, the operator $A_\Phi$ generates a holomorphic semigroup on $X$ which restricts to a strongly 
continuous  and holomorphic semigroup on $\overline{D}$.
\end{thm}

We prepare the proof of Theorem \ref{t.greiner} with some preliminary results.

\begin{lem}\label{l.1}
Assume that $\lambda \in \rho (A_0)$. Then $D= D(A_0)\oplus \ker (\lambda- A)$. 
\end{lem} 

\begin{proof}
If $u\in D(A_0)\cap \ker (\lambda - A)$, then $u\in D(A_0)$ satisfies $A_0u=\lambda u$. As $\lambda \in \rho (A_0)$ we must have
$u=0$. Now let $u \in D$ be arbitrary. Since $\lambda - A_0$ is surjective, we find $u_0 \in D(A_0)$ with $(\lambda - A)u = (\lambda - A_0)u_0$. Consequently $u-u_0 \in \ker (\lambda - A)$ whence $u= u_0 + (u-u_0) \in D(A_0) + \ker (\lambda - A)$. 
\end{proof}

In our framework we can formulate well-posedness of the following boundary value problem \eqref{eq.poisson}.

\begin{lem}\label{l.2}
Let $\lambda \in \rho (A_0)$. Then for every $h\in \partial X$ the problem
\begin{equation}\label{eq.poisson}
\begin{cases}
\lambda u - A u & = 0\\
Bu & = h
\end{cases}
\end{equation}
has a unique solution $u =\colon S_\lambda h$ in $D$. The operator $S_\lambda : \partial X \to D$ is continuous, 
$B S_\lambda  = I_{\partial X}$ and $S_\lambda B$ is the projection onto $\ker(\lambda - A)$ along $D(A_0)$.
\end{lem}

\begin{proof}
By Lemma \ref{l.1} the map $B$ defines a continuous bijection between $\ker (\lambda - A)$ and $\partial X$. As a consequence
of the open mapping theorem $S_\lambda \coloneqq (B|_{\ker (\lambda - A)})^{-1}$ is a continuous linear operator from $\partial X$
to $\ker (\lambda - A)$. Obviously, $u\coloneqq S_\lambda h$ solves \eqref{eq.poisson}. If $\tilde u$ was another solution, 
we must have $u-\tilde u \in \ker B\cap \ker (\lambda - A) = \{0\}$ by Lemma \ref{l.1}. This proves uniqueness.
The last assertions are obvious from the definition.
\end{proof}

\begin{lem}\label{l.3}
Let $\lambda \in \rho (A_0)$. Then for $u\in \overline{D}$ one has $u \in D(A_\Phi)$ if and only if $(I - S_\lambda\Phi)u \in D(A_0)$. In this case
\[
(\lambda - A_\Phi)u = (\lambda - A_0)(I- S_\lambda \Phi )u
\]
for every $u \in D(A_\Phi)$. In particular, if $(I-S_\lambda\Phi) : \overline{D} \to \overline{D}$ is invertible, we have $\lambda \in \rho (A_\Phi)$
and 
\begin{equation}\label{eq.resaphi}
R(\lambda, A_\Phi) = (I- S_\lambda \Phi)^{-1}R(\lambda, A_0).
\end{equation}
\end{lem}

\begin{proof}
Let us first assume that $u \in D(A_\Phi)$, i.e.\ $u \in D$ and $Bu=\Phi u$. Since $BS_\lambda  = I_{\partial X}$ by Lemma \ref{l.2},
we find $B(I- S_\lambda \Phi)u = Bu - BS_\lambda \Phi u = Bu - \Phi u=0$. Thus $(I-S_\lambda \Phi)u \in \ker B$ and consequently
$(I- S_\lambda\Phi)u \in D(A_0)$.

Conversely, if we assume that $u - S_\lambda \Phi u \in D(A_0)$, then $u = (I-S_\lambda \Phi)u + S_\lambda \Phi u \in D$, as
$S_\lambda$ takes values in $D$, and $Bu = BS_\lambda \Phi u = \Phi u$ since $BS_\lambda = I_{\partial X}$. Thus $u \in D(A_\Phi)$.

Let us now assume that $u \in D(A_\Phi)$ or, equivalently, that $(I-S_\lambda \Phi)u \in D(A_0)$. Then
\begin{align*}
(\lambda - A_0)(I- S_\lambda \Phi)u & = (\lambda - A)u - (\lambda - A)S_\lambda \Phi u = (\lambda - A)u 
\end{align*}
since $S_\lambda$ takes values in $\ker (\lambda - A)$. This implies \eqref{eq.resaphi}.
\end{proof}

We now obtain the following criterion to prove that $A_\Phi$ generates a holomorphic semigroup.

\begin{prop}\label{p.1}
Assume that there is some $\rho >\omega $ such that for $\lambda \in \CC$ with $\Re\lambda >\rho$ the map
$I-S_\lambda \Phi$ is invertible with
\[
C \coloneqq \sup_{\Re\lambda > \rho} \|(I-S_\lambda \Phi)^{-1}\|_{\cL (\overline{D})} < \infty.
\]
Then $A_\Phi$ generates a holomorphic semigroup on $X$. 
\end{prop}

\begin{proof}
Set
\[
M\coloneqq \sup_{\Re\lambda >\rho}\|\lambda R(\lambda, A_0)\| <\infty
\]
since  $A_0$ generates a holomorphic semigroup.
As a consequence of Lemma \ref{l.3}, for $\Re \lambda > \rho$ we have $\lambda \in \rho (A_\Phi)$ and
\[
\|\lambda R(\lambda, A_\Phi)\| = \| \lambda (I-S_\lambda\Phi)^{-1} R(\lambda, A_0)\| \leq \|(I-S_\lambda\Phi)^{-1}\|\|\lambda R(\lambda, A_0)\| \leq CM.
\]
By Theorem \ref{t.holgen}, this implies that $A_\Phi$ generates a holomorphic  semigroup on $X$. 
\end{proof}

We can now prove the main result of this section.

\begin{proof}[Proof of Theorem \ref{t.greiner}]
In view of Proposition \ref{p.1}, making use of the Neumann series, it suffices to prove that
$S_\lambda\Phi \to 0$ in $\cL(\overline{D})$ as $\Re\lambda \to \infty$. Since $\Phi: \overline{D}\to \partial X$ is compact it suffices to prove that $S_\lambda h \to 0$ as $\Re \lambda \to \infty$ for every $h \in \partial X$.

To prove this, let $h \in \partial X$ and fix $\mu\in \rho  (A_0)$. We put 
\[
u_\lambda \coloneqq S_\lambda h, \quad u_\mu \coloneqq S_\mu h \quad\mbox{and}\quad u = u_\lambda - u_\mu.
\]
Note that $u_\lambda, u_\mu, u \in D$ and that $u \in D(A_0)$. Since $u_\lambda \in \ker (\lambda - A)$ and $u_\mu \in \ker (\mu - A)$ we have
\[
(\lambda - A_0) u = -(\lambda - A)u_\mu = (\mu-\lambda )u_\mu 
\]
and hence $u = (\mu-\lambda)R(\lambda, A_0)u_\mu$. Consequently,
\[
u_\lambda = u_\mu -\lambda R(\lambda, A_0)u_\mu + \mu R(\lambda, A_0)u_\mu \to u_\mu - u_\mu + 0 =0
\]
as $\Re\lambda \to \infty$, since $\lambda R(\lambda, A_0)f \to f$ for every $f \in \overline{D(A_0)} = \overline{D}$.
\end{proof}

We can now establish some additional properties of the operator $A_\Phi$ and the semigroup generated by it. We start with compactness.

\begin{cor}\label{c.compact}
In the situation of Theorem \ref{t.greiner}, if $A_0$ has compact resolvent, then so does $A_\Phi$.
\end{cor}

\begin{proof}
This follows immediately from the identity \eqref{eq.resaphi} and the ideal property of compact operators.
\end{proof}

Next we address positivity of the semigroup. Most often we will be concerned with Banach lattices such as $C(\overline{\Omega})$
or $L^\infty(\Omega)$. However, we will occasionally (for example in the following corollaries) also consider closed subspaces of such spaces and therefore need the notion of positivity also in a more general setting. 
To that end, we assume that our Banach space $X$ is the complexification  of a real ordered Banach space $X_\CR$. This means that in the real Banach space $X_{\CR}$ a positive, proper, closed 
cone $X_+$ is given, i.e.\ we have $X_++X_+ \subset X_+$, $\CR_+\cdot X_+ \subset X_+$ and $X_+ \cap (-X_+) = \{0\}$.
For $u\in X$ we write $u\geq 0$ if $u\in X_+$. An operator $S : X \to X$ is called \emph{positive} if
$SX_+ \subset X_+$, we write $S \geq 0$. Given two operators $S_1, S_2 : X \to X$, we write $S_1\leq S_2$ if $S_2-S_1 \geq 0$. A semigroup $T$ on $X$ is called \emph{positive} if $T(t)\geq 0$ for all $t>0$.

If $Y \subset X$ is a closed subspace of $X$, then $Y_+ \coloneqq Y\cap X_+$ is a closed, proper cone, such that $Y_{\CR}\coloneqq
Y\cap X_{\CR}$ becomes an ordered Banach space. Note that we do not assume that our cone is generating, i.e.\ we do not necessarily have that $X_+-X_+ =X_{\CR}$.

\begin{cor}\label{c.positive}
Assume in addition to Hypothesis \ref{hyp.greiner} that $X$ is the complexification of a real ordered Banach space and that $A_0$ generates a positive semigroup. If there is a $\rho >\omega$ such that
for $\lambda \in \CR$ with $\lambda > \rho$ the operator $S_\lambda \Phi$ is positive, then also the semigroup generated
by $A_\Phi$ is positive.
\end{cor}

\begin{proof}
If the semigroup $T$ generated by $A_0$ is positive then we have $R(\lambda, A_0) \geq 0$ for $\lambda > \omega$, as the resolvent is given as the Laplace transform of the semigroup.
For sufficiently large $\lambda \in \CR$ we have $\|S_\lambda\Phi\| <1$ and $S_\lambda\Phi$ positive. Thus, by the Neumann series,
\[
(I-S_\lambda\Phi)^{-1} = \sum_{n=0}^\infty (S_\lambda\Phi)^n
\]
is a positive operator. It follows from \eqref{eq.resaphi} that $R(\lambda, A_\Phi)$ is positive for sufficiently large $\lambda$. It follows
from the Post--Widder inversion formula \cite[Theorem 1.7.7]{abhn}
that the semigroup generated by $A_\Phi$ is positive.
\end{proof}

Next we want to compare different perturbations of our operator $A$. We can obtain different perturbations by either using different boundary operators $B$ or by using different boundary perturbations $\Phi$.

\begin{cor}\label{c.compare}
Let $X, D, \partial X$ and $A$ be as in Hypothesis \ref{hyp.greiner} and assume that $X$ and $\partial X$ are complexifications of real ordered Banach spaces. Moreover, assume that we are given
maps $B_1, B_2 : D \to \partial X$ and $\Phi_1, \Phi_2 : \overline{D} \to \partial X$ such that Hypothesis \ref{hyp.greiner}
is satisfied for the operators $A, B_1, \Phi_1$ and the operators $A, B_2, \Phi_2$. We write $A_0^j \coloneqq A|_{\ker B_j}$ and
$S_\lambda^j \coloneqq (B_j|_{\ker (\lambda - A)})^{-1}$ for $j=1,2$. Finally, we assume that
\begin{enumerate}
[(a)]
\item The semigroup generated by $A_0^j$ is positive for $j=1,2$;
\item $0\leq \Phi_1 \leq \Phi_2$;
\item For some $\rho >\omega$ and all $\lambda > \rho$ we have $0 \leq S_\lambda^1 \leq S_\lambda^2$;
\item If $u \in D$ is positive, then $B_2u \leq B_1 u$.
\end{enumerate}
Then for the semigroups $T_1$ generated by $A^1_{\Phi_1}$ and $T_2$ generated by $A^2_{\Phi_2}$ we have
$0\leq T_1(t) \leq T_2(t)$ for all $t>0$.
\end{cor}

\begin{proof}
Let us first note that since the operators $\Phi_j$ and $S_\lambda^j$ are positive for $\lambda >\rho$ and $j=1,2$, it follows
from Corollary \ref{c.positive} that $T_1$ and $T_2$ are positive semigroups. It follows from (b) and (c) that
\[
(I-S_\lambda^1\Phi_1)^{-1} = \sum_{n=0}^\infty (S_\lambda^1\Phi_1)^n \leq \sum_{n=0}^\infty (S_\lambda^2\Phi_2)^n = (I-S_\lambda^2\Phi_2)^{-1}
\]
for all $\lambda >\omega$. Now fix $f \geq 0$ and $\lambda>\rho$. We put $u_j \coloneqq R(\lambda, A_0^j)f$.
Then $(\lambda - A)(u_1-u_2) = 0$ and $B_1u_1 = B_2u_2=0$. Using our assumption (d) and the fact that $u_1 \geq 0$, 
we see that 
\[
B_2(u_1-u_2) =  B_2u_1 - B_1u_1 \leq 0.
\] 
Consequently, as $u_1-u_2 = S^2_\lambda (B_2(u_1 - u_2))$ and $S_\lambda^2$ is positive $u_1-u_2 \leq 0$. This proves
$R(\lambda, A_0^1) \leq R(\lambda, A_0^2)$. Combining this with the above and Equation \eqref{eq.resaphi}, we find
\[
R(\lambda, A_{\Phi_1}^1) = (I-S_\lambda^1\Phi_1)^{-1}R(\lambda, A_0^1) \leq (I-S_\lambda^2\Phi_2)^{-1}R(\lambda, A_0^2) 
= R(\lambda, A_{\Phi_2}^2)
\]
for all sufficiently large $\lambda$. By the Post--Widder inversion formula \cite[Theorem 1.7.7]{abhn} it follows that
$T_1 \leq T_2$.
\end{proof}

Our last topic is the strong Feller property for the semigroup generated by the perturbed operator.

\begin{cor}\label{c.strongfeller}
Assume in addition to Hypothesis \ref{hyp.greiner} that 
$X=L^\infty(\Omega)$ and $\overline{D} = C(\overline{\Omega})$ for some open and bounded $\Omega\subset \CR^d$. 
If $A_0$ generates a strong Feller semigroup on $X$, then so 
does $A_\Phi$.
\end{cor}

\begin{proof}
By the proof of \cite[Corollary 5.8]{akk16} it suffices to prove that for sufficiently large $\Re\lambda$ the operator $R(\lambda, A_\Phi)$ is a strong Feller operator. But this follows
from \eqref{eq.resaphi}: The hypothesis implies that $R(\lambda, A_0)$ is a strong Feller operator, in particular it maps  $L^\infty(\Omega)$ to $C(\overline{\Omega})$. Since
$U \coloneqq (I-S_\lambda \Phi)^{-1}$ is a bounded linear operator on $C(\overline{\Omega})$ also
$R(\lambda, A_\Phi)$ maps $L^\infty(\Omega)$ to $C(\overline{\Omega})$. Moreover, if $f_n$ is a bounded sequence in $L^\infty(\Omega)$ converging pointwise almost everywhere to $f$, then $R(\lambda, A_0)f_n$ is a bounded sequence which 
converges pointwise to $R(\lambda, A_0)f$. Since
$U$ is bounded on $C(\overline{\Omega})$ we have for $x \in \overline{\Omega}$
\begin{align*}
 R(\lambda, A_\Phi)f_n(x) & = \langle UR(\lambda, A_0)f_n, \delta_x\rangle = \langle R(\lambda, A_0)f_n, U^*\delta_x\rangle 
\\ & \to  \langle R(\lambda, A_0)f, U^*\delta_x\rangle = R(\lambda, A_\Phi)f(x),
\end{align*}
where we have used dominated convergence.
\end{proof}

\section{Local Robin boundary conditions}\label{s.local}

In this section we collect some results on elliptic operators with local Robin boundary conditions which we will need in the next section when we establish our results concerning non-local boundary conditions.

Let $\Omega \subset \CR^d$ be a bounded open set with Lipschitz boundary. As we are talking about positive semigroups, 
we will consider real-valued spaces $L^p(\Omega), C(\overline{\Omega}), C_b(\Omega)$ and $B_b(\Omega)$ throughout. Only when we 
are concerned with holomorphic semigroups we need spaces of complex-valued functions, in which case we pass to the complexification of these spaces. Concerning the coefficients of our operator we make the following assumptions.

\begin{hyp}\label{hyp.coeff}
We are given bounded, real-valued, measurable functions $a_{ij}$, $b_j$, $c_j$, $d_0$ on $\Omega$ for $i,j=1, \ldots d$.
The diffusion coefficients
$a= (a_{ij})$ are assumed to be bounded and \emph{strictly elliptic}, i.e.\ there is a constant $\eta>0$ such that
for all $\xi \in \CR^d$ and almost all $x \in \Omega$ we have
\[
\sum_{i,j=1}^d a_{ij}(x)\xi_i\xi_j \geq \eta |\xi|^2.
\]
\end{hyp}

With these assumptions we define the operator $\cA : H^1(\Omega)\to \mathscr{D}(\Omega)'$ by
\[
\cA u \coloneqq -\sum_{i,j=1}^d D_i (a_{ij}D_j u) - \sum_{j=1}^d D_j(b_ju) + \sum_{j=1}^d c_j D_j u + d_0 u.
\]
Here, $H^1(\Omega)$ denotes the usual Sobolev space of order one, 
$\mathscr{D}(\Omega) = C_c^\infty(\Omega)$ is the space of all test functions and $\mathscr{D}(\Omega)'$ is the space of 
all distributions. We introduce the continuous bilinear form $\mathfrak{a} : H^1(\Omega)\times H^1(\Omega) \to \CR$ given by
\[
\mathfrak{a}[u,v] \coloneqq \sum_{i,j=1}^d \int_{\Omega} a_{ij}D_iuD_jv \, \dx x+ \sum_{j=1}^d \int_{\Omega}
b_j uD_jv + c_j (D_j u)v\, \dx x+ \int_{\Omega} d_0 uv\, \dx x
\]
for $u,v\in H^1(\Omega)$. Thus $\la \cA u, \varphi\ra = \mathfrak{a}[u,\varphi]$ for all $u\in H^1(\Omega)$ and $\varphi
\in \mathscr{D}(\Omega)$.

If $u\in H^1(\Omega)$, we say that $\cA u \in L^2(\Omega)$ if there exists a function $f \in L^2(\Omega)$ such that
$\la \cA u, \varphi \ra = [ f, \varphi]$ for all $\varphi \in \mathscr{D}(\Omega)$. Here, and in what follows, 
\[
[f, g]\coloneqq \int_{\Omega} f  g \, \dx x\]
denotes the scalar product in $L^2(\Omega)$. If $\cA u \in L^2(\Omega)$
the function $f$ above is unique and we identify $\cA u$ and $f$.

Next we define the \emph{weak conormal derivative} by testing against  functions in $H^1(\Omega)$ rather than functions in $\mathscr{D}(\Omega)$ only.

\begin{defn}
Let $u\in H^1(\Omega)$ be such that $\cA u \in L^2(\Omega)$. For a function $h\in L^2(\partial \Omega)$ we say 
that $h$ is the \emph{weak conormal derivative of $u$} and write $\partial_\nu^{\cA} u \coloneqq h$ if the Green formula
\[
\mathfrak{a}[u,v] - [\cA u, v] = \int_{\partial \Omega} hv\, \dx \sigma 
\]
holds for all $v \in H^1(\Omega)$.
\end{defn}

Under our assumptions on the coefficients the weak conormal derivative, if it exists, is unique.
It depends on the operator $\cA$ only through the coefficients $a=(a_{ij})$ and $b_j$. Moreover, if the coefficients and the boundary 
of $\Omega$ are smooth enough the weak conormal derivative coincides with the usual conormal derivative
\[
\partial_\nu^{\cA}u = \sum_{j=1}^d\Big(\sum_{i=1}^d a_{ij} D_i u + \tr b_j u\Big)\nu_j
\]
where $\nu = (\nu_1, \ldots, \nu_d)$ is the unit outer normal of $\Omega$. In particular, $\partial_\nu^{\cA}\one =
\sum_{j=1}^d \tr b_j \nu_j$.
For a proof of these facts and more information
we refer to \cite[Section 8.1]{A15}.\smallskip

Next we endow our differential operator with Robin boundary conditions, given through a real 
function $\beta \in L^\infty(\partial\Omega)$. For now, we do not (as in Hypothesis \ref{h.mu}) assume that $\beta \geq 0$, 
but this assumption will be used later on in Theorem \ref{thm:local-gen-res} to obtain analyticity of the semigroup via Gaussian estimates.

To define the differential operator with Robin boundary conditions, we employ the theory of bilinear forms, defining
 $\mathfrak{a}_\beta : H^1(\Omega)\times H^1(\Omega) \to \CR$ by
\[
\mathfrak{a}_\beta [u,v] \coloneqq  \mathfrak{a}[u,v] + \int_{\partial\Omega} \beta uv \,\dx\sigma.
\]
The associated operator $\cA_{\beta}^2$ on $L^2(\Omega)$ is given by
\begin{align*}
D(\cA_{\beta}^2) & \coloneqq \{ u \in H^1(\Omega) : \exists \, f \in L^2(\Omega) \mbox{ with } \mathfrak{a}_\beta[u,v]= [f, v]\, \, \forall\, v \in H^1(\Omega)\}\\
\cA_{\beta}^2 u & \coloneqq f.
\end{align*}
Testing against test functions we see that $\cA_{\beta}^2u=\cA u$ for all $u\in D(\cA_{\beta}^2)$. By
the definition of the weak conormal derivative we obtain the following description of the domain:
\[
D(\cA_{\beta}^2) = \{ u\in H^1(\Omega) : \cA u \in L^2(\Omega) \mbox{ and } \partial_\nu^{\cA} u + \beta \tr u = 0\}.
\]
Thus $\cA_{\beta}^2$ is the realization of $\cA$ with Robin boundary condition. We immediately obtain the following generation 
result.

\begin{prop}
Assume Hypothesis \ref{hyp.coeff} and let $\beta \in L^\infty(\partial \Omega)$. Then the operator $-\cA_{\beta}^2$ generates a positive, strongly continuous semigroup $T^2_\beta$ on $L^2(\Omega)$. 
\end{prop}

\begin{proof}
Using Lemma \ref{l.am} below and the fact that the trace is a compact operator from $H^1(\Omega)$ to $L^2(\partial \Omega)$,
we see that the form $\mathfrak{a}_\beta$ is \emph{elliptic}, i.e.\ there are constants $\alpha >0$ and $\omega \geq 0$ such that
\[
\mathfrak{a}_\beta [u,u] +\omega \|u\|_{L^2(\Omega)}^2 \geq \alpha \|u\|_{H^1(\Omega)}.
\]
Thus, by standard results from the theory of quadratic forms (\cite[Section 1.4]{ouh05}) $-\cA_{\beta}^2$ generates a holomorphic semigroup $T^2_\beta$. The positivity of $T^2_\beta$ follows from \cite[Theorem 2.6]{ouh05} noting
that $\mathfrak{a}_\beta[u^+,u^-]=0$ for all $u \in H^1(\Omega)$. 
\end{proof}

We next investigate when the semigroup $T_\beta^2$ is sub-Markovian. We will use the following lemma.

\begin{lem}\label{l.positive}
Let $g\in L^2(\Omega)$ and $h\in L^2(\partial\Omega)$ be such that
\begin{equation}
\label{eq.positive}
\int_\Omega g v \,\dx x + \int_{\partial\Omega} h v \,\dx \sigma \geq 0
\end{equation}
for all $0\leq v \in H^1(\Omega)$. Then $g\geq 0$ a.e.\ on $\Omega$ and $h\geq 0$ a.e.\ on $\partial \Omega$. Moreover, if in 
\eqref{eq.positive} identity holds for all $v \in H^1(\Omega)$, then $g=0$ a.e.\ on $\Omega$ and $h=0$ a.e.\ on $\partial\Omega$.
\end{lem}

\begin{proof}
By \eqref{eq.positive} we have $\int_\Omega g v\dx x \geq 0$ for all $0\leq v \in C_c^\infty(\Omega)$. Thus
$g\geq 0$ almost everywhere on $\Omega$. Given a  function $\varphi \in C(\partial \Omega)$, we find a sequence
$v_n \in C^\infty(\overline{\Omega})$ such that $v_n|_{\partial\Omega} \to \varphi$ in $C(\partial \Omega)$, $0\leq v_n \leq \|\varphi\|_\infty$ in $\Omega$ and such that $v_n$ is supported in a relatively open set $U_n\subset\overline{\Omega}$ with $U_n \supset U_{n+1}$
and $\bigcap_{n\in\CN} U_n = \partial \Omega$. Choosing $v= v_n$ in \eqref{eq.positive} and letting $n\to \infty$, we infer from dominated convergence that $\int_{\partial\Omega} h \varphi \dx \sigma \geq 0$. As $\varphi \in C(\partial\Omega)$ was arbitrary, the claim follows.
\end{proof}

\begin{prop}\label{p.submarkov}
Assume Hypothesis \ref{hyp.coeff} and let $\beta \in L^\infty(\partial\Omega)$. We additionally assume that $b_j \in W^{1,\infty}(\Omega)$ for $j=1, \ldots, d$.
\begin{enumerate}
[(a)]
\item The semigroup $T_\beta^2$ is sub-Markovian if and only if
\begin{equation}
\label{eq.sM1}
\sum_{j=1}^d D_j b_j \leq d_0\qquad \mbox{almost everywhere on $\Omega$ and}
\end{equation}
\begin{equation}
\label{eq.sM2}
\sum_{j=1}^d \tr (b_j)\nu_j + \beta \geq 0\qquad \mbox{almost everywhere on $\partial \Omega$.}
\end{equation}
\item The semigroup $T_\beta^2$ is Markovian if and only if
\begin{equation}
\label{eq.M1}
\sum_{j=1}^d D_j b_j = d_0\qquad \mbox{almost everywhere on $\Omega$ and}
\end{equation}
\begin{equation}
\label{eq.M2}
\sum_{j=1}^d \tr (b_j)\nu_j + \beta = 0\qquad \mbox{almost everywhere on $\partial \Omega$.}
\end{equation}

\end{enumerate}
\end{prop}

\begin{proof}
(a) The semigroup $T_\beta^2$ is sub-Markovian if and only if  the Beurling--Deny--Ouhabaz criterion holds, i.e.\
\[
\mathfrak{a}_\beta [u\wedge 1, (u-1)^+] \geq 0
\]
for all $u \in H^1(\Omega)$, see
\cite[Chapter 2]{ouh05} and  \cite[Corollary 2.8]{mvv05} or \cite{d16} for the case where the form is not necessarily accretive. Recall that for $u\in H^1(\Omega)$ the functions $u\wedge 1$ and $(u-1)^+$ also belong to
$H^1(\Omega)$ and
\[
 D_j (u\wedge 1)=  \one_{\{u<1\}} D_j u \quad\mbox{and}\quad D_j (u-1)^+ = \one_{\{ u> 1\}} D_j u.
\]
Thus $D_i (u\wedge 1) D_j (u-1)^+ = (u-1)^+ D_j (u\wedge 1) = 0$. We see that
\begin{align*}
& \mathfrak{a}_\beta [u\wedge 1, (u-1)^+]\\  
=  & \int_{\Omega} \sum_{j=1}^d b_j D_j(u-1)^+\, \dx x + \int_{\{ u>1\}} d_0 (u-1)^+\,\dx x
+ \int_{\partial\Omega} \beta ( u -1)^+\, \dx\sigma\\
=  & -\int_{\Omega} \sum_{j=1}^d (D_jb_j) (u-1)^+\, \dx x +  \int_{\partial\Omega}\sum_{j=1}^d b_j \nu_j ( u-1)^+ \, \dx \sigma\\
& \qquad\qquad +
\int_{\Omega} d_0 (u-1)^+\,\dx x
+ \int_{\partial\Omega} \beta ( u -1)^+\, \dx\sigma.
\end{align*}
The latter is positive if \eqref{eq.sM1} and \eqref{eq.sM2} hold whence $T_\beta^2$ is sub-Markovian in this case. This shows sufficiency of these two conditions. 

Conversely, if the semigroup $T_\beta^2$ is sub-Markovian, the Beurling--Deny--Ouhabaz criterion yields
\[
\int_{\Omega} \Big( d_0-\sum_{j=1}^d D_jb_j\Big)(u-1)^+\dx x + \int_{\partial\Omega} \Big(\sum_{j=1}^d b_j\nu_j +\beta\Big)(u-1)^+\dx\sigma \geq 0
\]
for all $u\in H^1(\Omega)$. Choosing $u=\one+v$ with $0\leq v \in H^1(\Omega)$, Lemma \ref{l.positive} shows that
\eqref{eq.sM1} and \eqref{eq.sM2} are valid.\medskip

(b) A Markovian semigroup is in particular sub-Markovian whence the inequalities \eqref{eq.sM1} and \eqref{eq.sM2} are satisfied.
If $T_\beta^2$ is sub-Markovian, then it is Markovian if and only if $\one \in \ker (-\cA_\beta^2)$. Note that
\[
-\cA \one = \sum_{j=1}^d D_j b_j - d_0.
\]
Thus \eqref{eq.M1} is necessary for $T_\beta^2$ to be Markovian. If \eqref{eq.M1} holds, then for $v \in H^1(\Omega)$ we have
\[
\mathfrak{a}[\one, v] -[\cA \one, v] = \sum_{j=1}^d \int_{\Omega} (b_j D_jv + d_0v)\,\dx x = \sum_{j=1}^d \int_{\partial\Omega}
b_j \nu_j v\,\dx\sigma, 
\]
where we used an integration by parts. Thus saying $\partial_\nu^{\cA}\one +\beta = 0$, i.e. $\one \in D(-\cA_\beta^2)$, is equivalent
to 
\[
\sum_{j=1}^d\int_{\partial\Omega} b_j\nu_j v \,\dx\sigma = - \int_{\partial\Omega}\beta v \,\dx\sigma
\]
for all $v\in H^1(\Omega)$ and hence to \eqref{eq.M2}.
\end{proof}

In order to apply the abstract results of Section \ref{s.2}, we need some results about the following
elliptic problem, which are also used implicitly in the proof of Theorem \ref{thm:local-gen-res}.
\begin{equation}\label{eq.weak-sol}
	\left\{\begin{aligned}
		\lambda u+\cA u & =f\text{ on }\Omega\\
		\partial_{\nu}^{\cA}u+\beta u & =h\text{ on }\partial\Omega.
		\end{aligned}\right.
\end{equation}
Obviously, $\mathfrak{a}_\beta$ defines a continuous sesquilinear mapping on $H^1(\Omega)$. By \cite[Corollary 2.5]{Dan09} it is also
\emph{elliptic}, i.e.\ there are some $\omega, \alpha >0$ such that $\mathfrak{a}_\beta[u,u] + \omega \|u\|_{L^2(\Omega)} ^2\geq \alpha \|u \|_{H^1(\Omega)}^2$. With this information at hand, 
one can prove existence and uniqueness of solutions to \eqref{eq.weak-sol} by means of the Lax--Milgram Theorem. Indeed, considering the 
continuous  functional $F$ on $H^1(\Omega)$, given by
$F(v) = \int_\Omega f v\, \dx x + \int_{\partial\Omega} hv\, \dx\sigma$, it follows from the Lax--Milgram Theorem
that for $\lambda > \omega $ there is a unique $u \in H^1(\Omega)$ such that
\[
\mathfrak{a}_\beta [u,v] + \lambda [u,v] = F(v)
\]
for all $v \in H^1(\Omega)$. From \cite[Theorem 3.14(iv)]{N11} we obtain the following result concerning regularity of the solution.

\begin{prop}
\label{prop:robin-reg}
Assume Hypothesis \ref{hyp.coeff}, fix $q>d$ and $\lambda>\omega$. Then there exist constants $\gamma>0$ and $C>0$ such that
whenever $f\in L^{q/2}(\Omega)$ and $h\in L^{q-1}(\partial\Omega)$ the unique solution $u$ of \eqref{eq.weak-sol} belongs
to  $C^\gamma(\overline{\Omega})$ and we have
\[
\|u\|_{C^\gamma(\overline{\Omega})} \leq C \big(\|f\|_{L^\frac{q}{2}(\Omega)} + \| h\|_{L^{q-1}(\partial\Omega)}\big).
\]
\end{prop}

The following lemma is easy to prove, see e.g.\ \cite[Lemma 2.3]{am}.

\begin{lem}
\label{l.am}
Let $X_1, X_2, X_3$ be Banach spaces such that $X_1$ is reflexive. Let $T: X_1 \to X_3$ be compact, $S: X_1 \to X_2$ be injective. 
Then, given $\eps>0$ there exists a constant $c>0$ such that
\[
\|Tx\|_{X_3} \leq \eps \|x\|_{X_1} + c\|Sx\|_{X_2}
\]
for all $x\in X_1$.
\end{lem}

We use this lemma to prove the following domination result.

\begin{prop}\label{p.domination}
Assume Hypothesis \ref{hyp.coeff} and let 
$\beta_1,\beta_ 2\in L^\infty(\partial \Omega)$ be such that $\beta_1\leq \beta_2$. There exists $\omega$ so 
that both $\mathfrak{a}_{\beta_1} +\omega$ and $\mathfrak{a}_{\beta_2}+\omega$
are coercive and such that for $\lambda >\omega$ the following holds. Let $0\leq f \in L^2(\Omega)$, $0\leq h \in L^2(\partial \Omega)$.
For $j=1,2$, let $u_j \in H^1(\Omega)$ be the unique solution of
\[
\left\{\begin{aligned}
		\lambda u+\cA u & =f\text{ on }\Omega\\
		\partial_{\nu}^{\cA}u+\beta_j u & =h\text{ on }\partial\Omega.
		\end{aligned}\right.
\]	
Then $0\leq u_2 \leq u_1$. 
\end{prop}
\begin{proof}
	We first show positivity for weak solutions $u$ of \eqref{eq.weak-sol}. To that end consider $f\leq 0$ and $h\leq 0$ for now. Since $u$ solves \eqref{eq.weak-sol} we have
	\[
		\lambda [u,v] + \mathfrak{a}_\beta [u,v] =[f, v] + \int_{\partial\Omega}hv\, \dx \sigma 
	\]
	for all $v\in H^1(\Omega)$. Setting $v:=u^+$ and noting that $\mathfrak{a}_\beta [u, u^+] = \mathfrak{a}_\beta [u^+, u^+]$ by the locality of $\mathfrak{a}_\beta$, we find
	\[
		\lambda [u^+,u^+]+\mathfrak{a}_\beta [u^+, u^+] =[f, u^+] +\int_{\partial\Omega}hu^+ \, \dx \sigma \leq 0.
	\]
	As $\mathfrak{a}_\beta + \omega$ is coercive we have that $\mathfrak{a}_\beta [u^+,u^+] + \omega\|u^+\|^2_{L^2(\Omega)}\geq\
	\alpha \|u^+\|^2_{H^1(\Omega)}$ for some $\alpha>0$. Together with $\lambda>\omega$ it follows that $\|u^+\|_{H^1(\Omega)}\leq 0$, whence $u\leq0$.\smallskip 
	
We can prove the domination similarly. This time we fix $f \geq 0$ and $h\geq 0$. The solution $u_j$ ($j=1,2$) satisfies the equation
\[
	\lambda [u_j, v] + \mathfrak{a}_{\beta_j}[u_j, v] = [f, v] + \int_{\partial \Omega} hv\, \dx \sigma
\]
for all $v\in H^1(\Omega)$. Subtracting these equations we find for a positive $v$ that
	\begin{align*}
		\lambda [u_2-u_1, v] + \mathfrak{a}[u_2-u_1,v]&=\int_{\partial\Omega}(\beta_1u_1-\beta_2u_2)v \, \dx\sigma \leq \int_{\partial\Omega} \beta_2(u_1-u_2)v\, \dx\sigma,
	\end{align*}
since $u_1 \geq 0$ by the above. Testing against $v\coloneqq (u_2-u_1)^+$, we find
	\begin{align*}
		& \lambda [(u_2-u_1)^+, (u_2-u_1)^+] + \mathfrak{a}[(u_2-u_1)^+, (u_2-u_1)^+]\\  \leq  & -\int_{\partial\Omega} \beta_2\big((u_2-u_1)^+\big)^2 \dx x\leq \|\beta_2\|_{L^\infty(\partial\Omega)} \int_{\partial\Omega} \big((u_2-u_1)^+\big)^2\dx\sigma.
	\end{align*}
Applying Lemma \ref{l.am} with $X_1= H^1(\Omega)$, $X_2= L^2(\Omega)$ and $X_3= L^2(\partial\Omega)$ where
$T: H^1(\Omega) \to L^2(\partial \Omega)$ is the trace operator (which is compact) and $S: H^1(\Omega) \to L^2(\Omega)$
is the natural embedding, given $\eps>0$ we find a constant $c>0$ such that
\begin{align*}
 & \|\beta_2\|_{L^\infty(\partial\Omega)} \int_{\partial\Omega} \big((u_2-u_1)^+\big)^2\dx\sigma\leq \eps \|(u_2-u_1)^+\|_{H^1(\Omega)}^2 + c \|(u_2-u_1)^+\|_{L^2(\Omega)}^2\\
 & \qquad = \eps \int_{\Omega} \big|\nabla (u_2-u_1)^+\big|^2\dx x + (c+\eps) \int_{\Omega} \big((u_2-u_1)^+\big)^2\dx x.
\end{align*}
Using the ellipticity of $\mathfrak{a}$ we deduce that, for a suitable constant $\alpha>0$, we have
\begin{align*}
&(\lambda +\alpha -\omega)\|(u_2-u_1)^+\|_{L^2(\Omega)}^2 +\alpha \int_{\Omega} \big|\nabla (u_2-u_1)^+\big|^2 \dx x\\ 
\leq & \eps \int_{\Omega} \big|\nabla (u_2-u_1)^+\big|^2 \dx x + (c+\eps)\|(u_2-u_1)^+\|_{L^2(\Omega)}^2
\end{align*}
Choosing $\eps = \alpha/2$ and $\lambda_0> \omega + c+\eps +1$, it follows that for $\lambda >\lambda_0$ we have
$(u_2-u_1)^+=0$, i.e.\ $u_2 \leq u_1$.
\end{proof}

Proposition \ref{p.domination} yields in particular the following monotonicity property.

\begin{cor}\label{c.monotone}
Assume Hypothesis \ref{hyp.coeff} and let $\beta_1, \beta_2 \in L^\infty(\Omega)$ be such that $\beta_1\leq \beta_2$. Then
$0\leq T_{\beta_2}^2(t) \leq T_{\beta_1}^2(t)$ for all $t\geq 0$.
\end{cor}

\begin{proof}
Proposition \ref{p.domination} shows that for large $\lambda$ we have $0\leq (\lambda + \cA_{\beta_2}^2)^{-1}
\leq (\lambda + \cA_{\beta_1}^2)^{-1}$. This implies the claim in view of Euler's formula.
\end{proof}

For our next result, we assume again that $0\leq \beta$ as in Hypothesis \ref{h.mu}. Under this assumption, we will show that
 the semigroup $T^2_\beta$ on $L^2(\Omega)$ always leaves the space $L^\infty(\Omega)$ invariant, even if $T_\beta^2$ 
 is not sub-Markovian. This follows from Gaussian estimates for the semigroup $T^2_\beta$ which can be proved under the assumption that $0\leq \beta$. It seems to be unknown whether this is necessary for the Gaussian estimates. As a second
 consequence of the Gaussian estimates, we see that the restriction $T_\beta$ of $T^2_\beta$
to $L^\infty(\Omega)$ is a holomorphic semigroup, by which we mean that the $\CC$-linear extension of $T^2_\beta|_{L^\infty(\Omega)}$ to the complexification $L^\infty(\Omega; \CC)$ of $L^\infty(\Omega)$ is holomorphic.

Of course the generator of $T_\beta$ is the part $A_{\beta}$ of $-\cA_{\beta}^2$ in $L^\infty(\Omega)$, i.e.
\begin{align*}
D(A_\beta) & = \{ u \in H^1(\Omega)\cap L^\infty(\Omega) : \cA u \in L^\infty(\Omega),
\partial_\nu^\cA u +\beta u =0 \}\\
A_\beta u& = -\cA u.
\end{align*}
We will also see that the semigroup $T^2_\beta$ has leaves the space $C(\overline{\Omega})$ invariant and restrincts to a strongly continuous semigroup on that space. Naturally,
 the generator of $T^C_\beta$ is the part $A_{\beta}^C$ of $-\cA_{\beta}^2$ in $C(\overline{\Omega})$, i.e.
\begin{align*}
D(A_{\beta}^C) & = \{ u \in H^1(\Omega)\cap C(\overline{\Omega}) : \cA u \in C(\overline{\Omega}),
\partial_\nu^\cA u +\beta \tr u =0 \}\\
A_{\beta}^C u& = -\cA u.
\end{align*}
As a consequence of the strong continuity of $T^C_\beta$ we find that $D(A_{\beta}^C)$ is dense in $C(\overline{\Omega})$.

\begin{thm}
\label{thm:local-gen-res}
Assume Hypothesis \ref{hyp.coeff} and let $0\leq \beta \in L^\infty(\partial\Omega)$. Then $T^2_\beta$ leaves the space
$L^\infty(\Omega)$ invariant.
Its restriction $T_\beta$ to $L^\infty(\Omega)$ is a holomorphic semigroup on $L^\infty(\Omega)$. Each operator
$T_\beta(t)$, $t>0$, is compact and enjoys the strong Feller property. In particular, $C(\overline{\Omega})$ is invariant. The restriction
$T^C_\beta$ of $T^2_\beta$ to $C(\overline{\Omega})$ is a strongly continuous and holomorphic semigroup.
\end{thm}

\begin{proof}
It was proved in \cite[Corollary 6.1]{Dan00} (see also \cite[Theorem 4.9]{AtE97}) that the semigroup $T^2_\beta$ has Gaussian estimates so that $T^2_\beta$ extrapolates to a consistent family of semigroups $T^q_\beta$ on $L^q(\Omega)$ for $q \in [1,\infty]$. In particular,
$T^2_\beta$ leaves the space $L^\infty(\Omega)$ invariant and restricts to a semigroup $T_\beta$ on this space.
By \cite[Theorem 5.3]{AtE97} the semigroup $T_\beta$ is holomorphic on $L^\infty(\Omega)$. Moreover, by the proof of
\cite[Theorem 4.3]{N11} $T_\beta(t)L^\infty(\Omega) \subset C(\overline{\Omega})$ for all $t>0$. It was also seen in that theorem
that $T_\beta (t)$ is compact for all $t>0$. We now show that $T_\beta(t)$ is strongly Feller for $t>0$. Since $T^2_\beta$ is ultracontractive
by \cite[7.3 Criterion (v)]{A04} it follows that $T^2_\beta (t) L^q(\Omega) \subset L^\infty(\Omega)$ and hence $T^2_\beta
(t) L^q(\Omega)
\subset T^2_\beta(t/2)L^\infty(\Omega) \subset C(\overline{\Omega})$ for some $q \in (2, \infty)$. By the closed graph theorem, 
$T^2_\beta(t)$ is a bounded operator from $L^q(\Omega)$ to $C(\overline{\Omega})$. Now the strong Feller property, as defined in Definition \ref{def.sf}, follows from the dominated convergence theorem. It follows from \cite[Theorem 4.3]{N11} that
the restriction of the semigroup to $C(\overline{\Omega})$ is strongly continuous.
\end{proof}

\section{Non-local boundary conditions}\label{s.non}

We are now prepared to prove the main results of this article. We begin by setting up the framework in which we apply
Greiner's boundary perturbation. In contrast to the last section, in this section only consider complex Banach spaces in order to handle (possibly) complex valued functions $\mu : \partial \Omega \to \mathscr{M}(\overline{\Omega})$. 

We assume throughout Hypotheses \ref{h.mu} and \ref{hyp.coeff}. In particular, we assume throughout that
$0\leq \beta \in L^\infty(\partial\Omega)$. We then define
\[
D \coloneqq \{ u \in C(\overline{\Omega}) \cap H^1(\Omega) : 
\cA u \in L^\infty(\Omega), \partial_\nu^{\cA} u \in L^p(\partial \Omega)\},
\]
where $p>d-1$ is as in Hypothesis \ref{h.mu}(b).
Endowed with the norm
\[
\|u\|_D \coloneqq \|u\|_{C(\overline{\Omega})} + \|u\|_{H^1(\Omega)} + \| \cA u\|_{L^\infty(\Omega)} + \|\partial_{\nu}^{\cA} u\|_{L^p(\partial\Omega)}
\]
$D$ is a Banach space which is continuously embedded into $X= L^\infty(\Omega)$. Since $D(A_\beta^C) \subset D$, it follows from Theorem \ref{thm:local-gen-res} that
$D$ is dense in $C(\overline{\Omega})$.
We define our maximal operator $A : D \to X$
by $A u \coloneqq - \cA u$ which is linear and continuous. 
We set $\partial X \coloneqq L^p(\partial \Omega)$ and consider the boundary operator 
$B: D \to \partial \Omega$ defined via $B u = \partial_\nu^{\cA} u + \beta u$ where
$\beta$ is as in Hypothesis \ref{h.mu}. Finally, given $\mu$ as in Hypothesis \ref{h.mu}, 
the function $\Phi: \overline{D} \to \partial X$ is
given by
\[
(\Phi u)(z) \coloneqq \int_{\overline{\Omega}} u(x)\mu (z)(\dx x).
\]

Making use of the results of Section \ref{s.2} we can now prove our main generation result for the operator $A_{\beta, \mu}$, defined by
\begin{align*}
D(A_{\beta, \mu}) & = \big\{ u \in C(\overline{\Omega})\cap H^1(\Omega) : \cA u \in L^\infty(\Omega), 
\partial_\nu^{\cA} u  + \beta u = \langle u, \mu (\cdot)\rangle \big\}\\
A_{\beta, \mu} & = -\cA u.
\end{align*}

The following result contains Theorem \ref{t.laplace} from the introduction as a special case.

\begin{thm}\label{t.main}
Assume Hypotheses \ref{h.mu} and \ref{hyp.coeff}. Then the operator $A_{\beta, \mu}$ generates a holomorphic semigroup
$T_{\beta, \mu}$ on $L^\infty(\Omega)$ which satisfies the strong Feller property. In particular, it leaves the space $C(\overline{\Omega})$ invariant. Its restriction to this space is a strongly continuous and holomorphic semigroup whose generator is
$A_{\beta, \mu}^C$, the part of $A_{\beta, \mu}$ in $C(\overline{\Omega})$.
\end{thm}

\begin{proof}
Noting that the operator $A_{\beta, \mu}$ is exactly the perturbed operator $A_\Phi$, where $A$ and $\Phi$ are as defined above,
the claim follows immediately from Theorem \ref{t.greiner} and Corollary \ref{c.strongfeller} once we verified that the maps $A, B$ and $\Phi$ satisfy Hypothesis 
\ref{hyp.greiner}.\smallskip

(a) The operator $B: D \to \partial X$ is surjective.

Fix $\lambda>\omega$. Given $h \in \partial X = L^p(\partial \Omega)$, it follows from Proposition \ref{prop:robin-reg} that the unique solution
$u\in H^1(\Omega)$ of the problem
\[
\begin{cases}
\lambda u + \cA u & = 0\\
\partial_{\nu}^{\cA} u + \beta u & = h
\end{cases}
\]
belongs to $C(\overline{\Omega})$. Moreover, $\cA u = -\lambda u \in C(\overline{\Omega})\subset L^\infty(\Omega)$. Thus,
$u \in D$ and $B u = h$, proving that $B$ is surjective.\medskip

(b) The boundary map $\Phi$ is compact.

Let $(u_n)_{n\in \CN}$ be a bounded sequence in $C(\overline{\Omega})$, say $\|u_n\|_{C(\overline{\Omega})} \leq M$
for all $n\in \CN$. Since $\mu (z) \ll \tau$ by Hypothesis \ref{h.mu}(c), for every $z \in \partial \Omega$ we find a Radon--Nikodym density $\varphi_z \in L^1(\overline{\Omega}, \tau)$ of $\mu(z)$ with respect to $\tau$, i.e.\ we have
\[
\int_{\overline{\Omega}} f(x)\,  \mu (z)(\mathrm{d} x) = \int_{\overline{\Omega}} f \varphi_z\, \dx\tau
\]
for all $f \in C(\overline{\Omega})$. In particular, $(\Phi u_n)(z) = \langle u_n, \varphi_z \rangle_{L^\infty (\tau), L^1(\tau)}$. 
Since the sequence $u_n$ is bounded in $L^\infty(\tau)$ and $L^1(\tau)$ is separable, it follows from the Banach--Alaoglu theorem
that we find a weak$^*$-convergent subsequence, say $u_{n_k} \rightharpoonup^* u$ for some $u \in L^\infty (\tau)$. In particular,
\[
(\Phi u_{n_k})(z) = \int_{\overline{\Omega}} u_{n_k} \varphi_z\,\dx\tau \to \int_{\overline{\Omega}} u\varphi_z\,\dx\tau
\]
for all $z \in \partial \Omega$, i.e.\ $\Phi u_n$ has a subsequence which converges pointwise.
Note that we have
\[
|(\Phi u_n)(z)| \leq M \|\mu (z)\|. 
\]
As a consequence of Hypothesis \ref{h.mu}(b) the functions $\Phi u_n$ have a $p$-integrable majorant and it follows from the dominated convergence theorem that $\Phi u_n$ has a subsequence which converges in $L^p(\partial \Omega)$.\medskip

(c) The operator $A_0$ is exactly the part of $-\cA_{\beta}^2$ in $L^\infty(\Omega)$. It follows from Theorem \ref{thm:local-gen-res}
that $A_0$ generates an holomorphic semigroup on $X=L^\infty(\Omega)$ which enjoys the strong Feller property and whose domain
is dense in $C(\overline{\Omega})$.
\end{proof}

We next prove some additional properties of the semigroup $T_{\beta, \mu}$ making use of the corollaries to Theorem \ref{t.greiner}.

\begin{prop}\label{p.properties}
Assume Hypotheses \ref{h.mu} and \ref{hyp.coeff} and let $T_{\beta, \mu}$ be the semigroup generated
by $A_{\beta, \mu}$ according to Theorem \ref{t.main}.
\begin{enumerate}
[(a)]
\item $T_{\beta, \mu}$ is compact.
\item If $\mu(z)$ is a positive measure for almost every $z\in \partial \Omega$, then the semigroup $T_{\beta, \mu}$ is positive.
\end{enumerate}
\end{prop}

\begin{proof}
(a) Follows immediately from Corollary \ref{c.compact}, noting that the semigroup generated by $A_0$ is compact as a consequence
of Theorem \ref{thm:local-gen-res}.\smallskip

(b) By Theorem \ref{thm:local-gen-res}, the semigroup generated by $A_0$ is positive.
If $\mu (z)$ is positive for almost every $z\in \partial \Omega$, then the map $\Phi$ is positive. Note that for the solution map
$S_\lambda$ the function $S_\lambda h$ is the unique solution of the boundary value problem
\[
\begin{cases}
\lambda u + \cA u & = 0\\
\partial_\nu^{\cA} u + \beta u & = h.
\end{cases}
\]
Thus, by Proposition \ref{p.domination}, $S_\lambda$ is positive for $\lambda >\omega$. Altogether $S_\lambda \Phi$ is positive
and it follows from Corollary \ref{c.positive} that $T_{\beta, \mu}$ is positive.
\end{proof}

Next we characterize when $T_{\beta, \mu}$ is Markovian.

\begin{prop}\label{p.further}
Assume in addition to Hypotheses \ref{h.mu} and \ref{hyp.coeff} 
that $\mu (z)$ is a positive measure for almost every  $z\in \partial\Omega$.
The following are equivalent.
\begin{enumerate}
[(i)]
\item The semigroup $T_{\beta, \mu}$ is Markovian.
\item  We have
\begin{equation}
\label{eq.mubetaM1}
\sum_{j=1}^d D_j b_j = d_0 \qquad \mbox{almost everywhere on $\Omega$ and}
\end{equation}
\begin{equation}
\label{eq.mubetaM2}
\mu(z)(\overline{\Omega}) = \beta(z) + \sum_{j=1}^d \nu_j(z) b_j(z) \qquad \mbox{ for almost all $z\in \partial\Omega$}.
\end{equation}
\end{enumerate}
\end{prop}

\begin{proof}
 Since $T_{\beta, \mu}$ is positive, (i) is equivalent to $\one \in \ker A_{\beta, \mu}$. Observe that $-\cA \one = \sum_{j=1}^d D_j b_j - d_0$. Thus $-\cA\one =0$ if and only if \eqref{eq.mubetaM1} holds. In that case, integration by parts yields for $v\in H^1(\Omega)$
 that
 \[
 \mathfrak{a}[\one, v] -[\cA \one, v] = \sum_{j=1}^d \int_\Omega b_jD_j v +d_0 v\dx x = \sum_{j=1}^d \int_{\partial\Omega}
 b_j \nu_j v\, \dx \sigma .
 \]
 Thus $\one \in D(A_{\beta, \mu})$ if and only if
 \[
 \sum_{j=1}^d\int_{\partial\Omega} b_j(z)\nu_j(z)v(z)\, \dx\sigma(z) = \int_{\partial\Omega}\big(-\beta (z) +\langle \mu (z), \one\rangle\big)
 v(z)\,\dx\sigma 
 \]
 for all $v\in H^1(\Omega)$. This is equivalent to \eqref{eq.mubetaM2}.
\end{proof} 

If we merely have inequalities in \eqref{eq.mubetaM1} and \eqref{eq.mubetaM2}, then the semigroup is sub-Markovian 
as we show next. In the proof, we use the following monotonicity result.

\begin{prop}\label{p.monotone}
Assume Hypothesis \ref{hyp.coeff} and let $\beta_1, \beta_2 \in L^\infty(\partial \Omega)$ 
with $\beta_2 \leq \beta_1$. Moreover, let functions
$\mu_1, \mu_2 : \partial \Omega \to \mathscr{M}(\overline{\Omega})$ be given such that $0\leq \mu_1(z) \leq \mu_2(z)$
for almost all $z\in \partial\Omega$ and such that $\mu_1, \mu_2$ satisfy Hypothesis \ref{h.mu} with the same $p$.
Then
\[
0 \leq T_{\beta_1, \mu_1}(t) \leq T_{\beta_2, \mu_2}(t)
\]
for all $t \geq 0$. 
\end{prop}

\begin{proof}
The semigroups $T_{\beta_1, \mu_1}$ and $T_{\beta_2, \mu_2}$ are obtained from the same maximal operator $A$ but using
different boundary perturbations $\Phi_j \colon u \mapsto  \langle \mu_j (\cdot), u\rangle$ and boundary operators $B_j \colon 
u \mapsto \partial_\nu^\cA u + \beta_j u$. We clearly have $B_2 u \leq B_1 u$ and $0\leq \Phi_1 u \leq \Phi_2 u$ for $u\geq 0$. Moreover, if we write
$S_\lambda^j \coloneqq (B_j|_{\ker (\lambda - A)})^{-1}$, then we have $S_\lambda^1 \leq S_\lambda^2$ by Proposition \ref{p.domination}. Thus Corollary \ref{c.compare} yields the claim. 
\end{proof}

\begin{prop}\label{p.evenmore}
Assume in addition to Hypotheses \ref{h.mu} and \ref{hyp.coeff} that $\mu(z)$ is positive for almost all $z \in \partial\Omega$
and that $b_j \in W^{1,\infty}(\Omega)$ for $j=1, \ldots, d$. If
\begin{equation}
\label{eq.mubetasM1}
\sum_{j=1}^d D_j b_j \leq d_0 \qquad \mbox{almost everywhere on $\Omega$ and }
\end{equation}
\begin{equation}
\label{eq.mubetasM2}
\mu (z)(\overline{\Omega}) \leq \beta (z) + \sum_{j=1}^d \tr (b_j)(z)\nu_j (z) \qquad \mbox{for almost all $z\in \partial\Omega$}
\end{equation}
then the semigroup $T_{\beta,\mu}$ is sub-Markovian.
\end{prop}
\begin{proof}
Assume at first that $\sum_{j=1}^d D_j b_j = d_0$. Let us define  $\beta_0(z) \coloneqq \mu (z)(\overline{\Omega})-\sum_{j=1}^d \tr b_j (z) \nu_j(z)$. 
By Proposition \ref{p.further} the semigroup $T_{\beta_0, \mu}$ is Markovian. As a consequence of Proposition \ref{p.monotone} 
we have $0 \leq T_{\beta, \mu}(t) \leq T_{\beta_0, \mu}(t)$ for all $t>0$ which clearly implies that $T_{\beta, \mu}$ is sub-Markovian. That $T_{\beta, \mu}$ is still sub-Markovian when $\sum_{j=1}^d D_j b_j \leq d_0$ 
follows from a standard perturbation result:

Denote by $\tilde A_{\beta, \mu}$ the operator where $d_0$ is replaced by $\tilde d_0 \coloneqq \sum_{j=1}^d D_j b_j$. Then the semigroup 
$\tilde T_{\beta, \mu}$ generated by $\tilde A_{\beta, \mu}$ is sub-Markovian by what has been proved so far. Note that
$A_{\beta, \mu} + (d_0- \tilde d_0) = \tilde A_{\beta, \mu}$, so that $\tilde A_{\beta, \mu}$ is a bounded  and positive perturbation of $A_{\beta, \mu}$. Using a perturbation result for resolvent positive operators \cite[Proposition 3.11.12]{abhn} we find that 
$R(\lambda, A_{\beta, \mu}) \leq R(\lambda, \tilde{A}_{\beta, \mu})$ for large enough $\lambda$ and the domination of the semigroups
follows from the Post--Widder inversion formula \cite[Theorem 1.7.7]{abhn}. Alternatively, the domination property can be inferred from
the Dyson--Phillips formula for the perturbed semigroup, see \cite[Example 3.4]{k13a} for a version which covers our setting.
\end{proof}

As a further consequence of Proposition \ref{p.monotone} we have 
\begin{equation}\label{eq.monotone}
 0\leq T_{\beta, 0}(t) \leq T_{\beta, \mu}(t)
\end{equation}
for all $t>0$  in the case where $\mu(z)$ is a positive measure for almost every $z\in \partial \Omega$. 
We note that for $\mu \equiv 0$ we have $T_{\beta, 0}(t) = T_\beta (t)$, where $T_\beta$ is the semigroup on 
$L^\infty(\Omega)$, defined in Section \ref{s.local} for local Robin boundary conditions.
It thus follows from Proposition \ref{p.submarkov} that condition 
\eqref{eq.mubetasM1} is necessary for $T_{\beta, \mu}$ to be sub-Markovian. It seems not so easy to show that also condition 
\eqref{eq.mubetasM2} is necessary for this. Also concerning the positivity of the semigroup $T_{\beta, \mu}$ it seems unclear if the condition that $\mu (z)$ is a positive measure for almost every $z \in \partial \Omega$ is necessary.
However, in Section \ref{s.6} we will give a proof of necessity in the special case where every measure
$\mu (z)$ is absolutely continuous with respect to the Lebesgue measure.

\section{Asymptotic behavior}

Throughout this section we assume Hypotheses \ref{h.mu} and \ref{hyp.coeff} so that $T_{\beta, \mu}$ is a semigroup on $L^\infty(\Omega)$. It is our aim to describe its asymptotic behavior as $t\to \infty$. Since $T_{\beta,\mu}(t)L^\infty(\Omega)
\subset C(\overline{\Omega})$ for all $t>0$ it suffices to study $T_{\beta, \mu}^C$, the restriction to $C(\overline{\Omega})$, which is a strongly continuous semigroup. We also assume throughout that $\mu (z) \geq 0$ for almost all $z\in \partial \Omega$ so that the semigroup is positive.

For the definition of spectral bound and irreducibility we refer to Appendix \ref{appendix}.
The asymptotic behavior of $T_{\beta, \mu}^C$ is determined by the spectral bound $s(A_{\beta, \mu}^C)$ of its generator (see Appendix \ref{appendix}). 
We first show that the spectrum is not empty.

\begin{prop}\label{p.nonempty}
Assume that $\mu (z)\geq 0$ for almost all $z\in \partial \Omega$. Then $s(A_{\beta, \mu}^C) > -\infty$. Moreover, $s(A_{\beta, \mu}^C)$ is an eigenvalue of $A_{\beta,\mu}^C$ with positive eigenfunction.
\end{prop}

\begin{proof}
We first show that $s(A_{\beta, 0}^C) \leq s(A_{\beta, \mu}^C)$.  As a consequence of Proposition \ref{p.monotone}
we have $0\leq T_{\beta,0}^C(t) \leq T_{\beta, \mu}^C(t)$. Taking Laplace transforms, it follows that
$0\leq R(\lambda, A_{\beta,0}^C) \leq R(\lambda, A_{\beta, \mu}^C)$ for all large enough $\lambda$. By \cite[Theorem 5.3.1]{abhn}
for a positive semigroup  the abscissa of the Laplace transform coincides with the spectral bound. Thus, if we assume that
$s(A_{\beta, 0}^C) > s(A_{\beta, \mu}^C)$ we have $0\leq R(\lambda, A_{\beta,0}^C) \leq R(\lambda, A_{\beta, \mu}^C)$
for all $\lambda > s(A_{\beta, 0}^C)$. By \cite[Proposition 3.11.2]{abhn} we have $s(A_{\beta,0}^C) \in \sigma (A_{\beta, 0}^C)$ and hence $\sup_{\lambda >s(A_{\beta, 0}^C)}\|
R(\lambda, A_{\beta,0}^C)\|=\infty$. Consequently, also $\|R(\lambda, A_{\beta,\mu}^C)\|$ is unbounded as $\lambda \downarrow
s(A_{\beta,0}^C)$. It thus follows that $s(A_{\beta, 0}^C) \in \sigma (A_{\beta, \mu}^C)$, a contradiction to our assumption
$s(A_{\beta_0}^C) > s(A_{\beta, \mu}^C)$.

The operator
$A_{\beta, 0}^C$ is the part of $-\cA_{\beta}^2$ in $C(\overline{\Omega})$, defined before Theorem \ref{thm:local-gen-res}. It follows from 
Proposition \ref{p.irreducible} that the semigroup generated by $-\cA_\beta^2$ is irreducible. Since the resolvent of that operator is compact, it follows from de Pagter's Theorem (see \cite[Theorem 3]{dp86} or \cite[C-III Theorem 3.7.(c)]{schreibmaschine}) that $s(-\cA_\beta^2)>-\infty$. But we have
$s(A_{\beta, 0}^C) = s(-\cA_{\beta}^2)$ since the resolvents are compact and consistent, see \cite[Proposition 2.6]{a94}.
\end{proof}

Note that the semigroup $T_{\beta, \mu}^C$ is compact and hence immediately norm continuous whence spectral bound and growth bound coincide.
Thus, if $s(A_{\beta, \mu}) < 0$, then $\|T_{\beta, \mu}^C(t)\|\leq Me^{-\eps t}$ for all $t>0$ and suitable constants $M> 0, \eps>0$, i.e.\ the semigroup is exponentially stable. If, on the other hand, $s(A_{\beta, \mu}^C) >0$ then there exists $\eps >0$ $M>0$ such that $\| T_{\beta, \mu}^C(t)\| \geq Me^{\eps t}$ for all $t>0$. Finally, if $s(A_{\beta, \mu}) = 0$, then the semigroup converges if it is bounded. This is not easy to decide, though. However, we have a precise criterion for the semigroup to be 
sub-Markovian. In that case, we obtain the following result from Theorem \ref{t.asympt}.

\begin{prop}\label{p.conv}
Assume that $\mu(z) \geq 0$ and
\begin{equation}\label{eq.sub1}
\mu (z)(\overline{\Omega}) \leq \beta (z) + \sum_{j=1}^d \tr b_j \nu_j (z)
\end{equation}
for almost every $z\in \partial \Omega$ and
\begin{equation}
\label{eq.sub2}
\sum_{j=1}^d D_j b_j \leq d_0
\end{equation}
almost everywhere. 
Then there exist a positive projection $P\in \mathscr{L}(C(\overline{\Omega}))$ with finite rank and $M> 0$, $\eps >0$ such that 
\[
\|T_{\mu, \beta}^C(t) - P\|_{\mathscr{L}(C(\overline{\Omega}))} \leq Me^{-\eps t}
\]
for all $t>0$.
\end{prop}
In the situation of Proposition \ref{p.conv}, if $s(A_{\beta, \mu}^C)=0$, there exists a function $0<u = Pu$, i.e.\ a positive function in the kernel of $A_{\beta, \mu}^C$. If the semigroup is Markovian, then $\one$ is such a function. It is interesting to know when it is the only one (up to a scalar multiple).
If $T_{\beta, \mu}^C$ is irreducible, then this is the case. Unfortunately, it is not easy to prove irreducibility on $C(\overline{\Omega})$.
However, it follows from the domination property \eqref{eq.monotone}  that $T_{\beta, \mu}^C$ is irreducible whenever $T_{\beta, 0}^C$ is so. As for the latter semigroup,
 a particular case will be settled in Theorem \ref{t.convergence}. 
 We also remark that in a forthcoming paper \cite[Section 6]{AtEff} it will be shown that
 $T_{\beta, 0}^C$ is irreducible whenever $\Omega$ is connected, $b_j = 0$ and $a_{ij}=a_{ji}$ for $i,j=1, \ldots, d$.
 
 \begin{thm}\label{t.conv2}
 Assume that $\mu (z) \geq 0$ and 
 \[
 0\leq \mu (z) (\overline{\Omega}) = \beta (z) + \sum_{j=1}^d \tr b_j \nu_j (z)
 \]
 for almost all $z\in \partial \Omega$ and $\sum_{j=1}^d D_j b_j = d_0$. Assume further that $T_{\beta, 0}^C$ is irreducible. Then there exist a strictly positive measure $\rho$ on $\overline{\Omega}$ and constants $\eps, M >0$ such that
 for $P \in \mathscr{L}(C(\overline{\Omega}))$, given by
 \[
 Pf = \int_{\overline{\Omega}} f\,\dx \rho \cdot \one
 \]
 for all $f\in C(\overline{\Omega})$, we have
 \[
 \|T_{\beta, \mu}^C (t) - P\|_{\cL (C(\overline{\Omega}))} \leq M e^{-\eps t}
 \]
 for all $t>0$.
 \end{thm}
 
 \begin{proof}
 By Proposition \ref{p.properties} the semigroup $T_{\beta, \mu}^C$ is Markovian and hence $\one$ is a fixed vector of the semigroup.
 As a consequence of \eqref{eq.monotone}, $T_{\beta, \mu}^C$ is irreducible. Now the claim follows from Theorem \ref{t.asymp2}.
 \end{proof}

We next prove exponential stability in the sub-Markovian case.

\begin{thm}\label{t.sub}
Assume that $\mu(z)\geq 0$ for almost all $z\in \partial \Omega$ and that \eqref{eq.sub1} and \eqref{eq.sub2} hold.
Moreover, assume that $T_{\beta, 0}^C$ is irreducible. If in \eqref{eq.sub1} or \eqref{eq.sub2} the 
inequality is strict on some set of positive measure, then there exist $\eps, M >0$ such that
\[
\|T_{\beta, \mu}^C(t)\|_{\mathscr{L}(C(\overline{\Omega}))} \leq Me^{-\eps t}
\]
for all $t>0$.
\end{thm}

\begin{proof}
Let us put 
\[
\tilde \beta (z) \coloneqq \mu (z)(\overline{\Omega}) - \sum_{j=1}^d \tr b_j (z) \nu_j (z)
\]
and $\tilde d_0(x) = \sum_{j=1}^d (D_j b_j)(x)$. Replace $d_0$ with $\tilde d_0$ and $\beta$ with $\tilde \beta$ and denote by $\tilde T_{\beta, \mu}$ the corresponding semigroup on $C(\overline{\Omega})$. We denote the generator of $\tilde{T}^C_{\tilde \beta, \mu}$
by $\tilde A_{\tilde \beta, \mu}^C$. Then $0 \leq T_{\beta, \mu}^C(t) \leq \tilde T^C_{\tilde\beta, \mu}(t)$
for all $t>0$ by Proposition \ref{p.monotone} and a perturbation argument, cf.\ the proof of Proposition \ref{p.evenmore}. By Proposition \ref{p.further} the semigroup $\tilde T$ is Markovian
so that its generator has spectral bound $0$. However, the generators of these two semigroups are different.
To see this, let us first assume that $\beta\neq \tilde \beta$ in $L^\infty(\partial\Omega)$. Note that the conormal derivative
$\partial_\nu^\cA =\partial_\mu^{\tilde{\cA}}$ does not depend on the zero order term $d_0$ resp.\ $\tilde d_0$. We find 
\[
\langle \one, \mu (z)\rangle = \partial_\nu^{\tilde \cA} \one + \tilde \beta \one \neq \partial_\nu^\cA \one + \beta \one.
\]
Thus $\one \not\in D(A_{\beta, \mu}^C)$ but $\one \in D(\tilde A_{\beta, \mu}^C)$. If, on the other hand, $\beta = \tilde \beta$
in $L^\infty (\partial \Omega)$, then we have $d_0\neq \tilde d_0$ in $L^\infty (\Omega)$. Note that $A_{\beta, \mu}\one = \tilde d_0 - d_0$. If $\tilde d_0 - d_0 \in C(\overline{\Omega})$, it follows that $\one \in D(A_{\beta, \mu}^C)$ but $A_{\beta, \mu}^C\one
\neq \tilde A_{\tilde \beta, \mu}^C\one$. If $\tilde d_0 - d_0 \not\in C(\overline{\Omega})$, then $\one \not\in D(A_{\beta, \mu}^C)$.
In any case we have $\tilde A^C_{\tilde\beta, \mu} \neq A^C_{\beta, \mu}$. Thus the claim follows from Theorem \ref{t.monotone}.
\end{proof}

Next we show a blow-up result in the case where we perturb a Markovian semigroup $T_{\beta, 0}$ by a positive $\mu$. Recall 
from Proposition \ref{p.submarkov} that $T_{\beta, 0}$ is Markovian if and only if the identities \eqref{eq.M1} and \eqref{eq.M2} hold.

\begin{thm}\label{t.explode}
Assume the identities \eqref{eq.M1} and \eqref{eq.M2} and that $\Omega$ is connected. 
If $\mu (z) \geq 0$ for almost all $z\in \partial \Omega$ but not identically $0$ almost everywhere, then there exist $\omega, M >0$ such that
\[
\| T_{\beta, \mu}^C(t)\|_{\mathscr{L}(C(\overline{\Omega}))} \geq M e^{\omega t}
\]
for all $t>0$.
\end{thm}

\begin{proof}
The semigroup $T_{\beta,0}^C$  is Markovian (by Proposition \ref{p.submarkov}) and has an extension
to $L^2(\Omega)$ which is irreducible (as a consequence of Proposition \ref{p.irreducible}). From Proposition \ref{p.zug}, 
it follows that $T_{\beta,0}^C$ is irreducible. By Proposition \ref{p.monotone} we have $T_{\beta,0}^C(t) \leq T_{\beta,\mu}^C(t)$ for all $t>0$.
Since $\partial_\nu\one + \beta\one =0 < \mu (z)(\Omega)$ for $z$ in a set of positive measure, one has
$\one\not\in D(A^C_{\beta, \mu})$. Thus the two semigroups are different and it follows from Theorem \ref{t.monotone} that $0=s(A_{\beta,0}^C) < s(A_{\beta,\mu}^C) =\colon \omega$.
Thus there exists $u \in C(\overline{\Omega})$ such that $u\geq \one$ with $A_{\beta,\mu}^Cu = \omega u$. But this implies
$T_{\beta,\mu}^C(t) u = e^{\omega t} u$ which, in turn, yields the claim.
\end{proof}

\begin{rem}
In particular, it follows from Theorem \ref{t.explode} that the only realization of our operator with non-local \emph{Neumann} boundary 
conditions (i.e.\ where $\beta =0$) which generates a sub-Markovian semigroup is that with classical (local) Neumann boundary conditions (i.e.\ $\beta =0$ and $\mu =0$).
\end{rem}

\section{Absolutely continuous measures \texorpdfstring{$\mu (z)$}{}}

In this section we consider the case where the measures $\mu(z)$ are absolutely continuous with respect to the Lebesgue measure on $\Omega$. More precisely, we assume that we are given a function $h\in L^2(\partial\Omega\times \Omega)$ such that
\[
\mu (z) (A) = \int_{A} h(z, x)\,\dx x.
\] 
In this situation we can use form methods to show that the semigroup $T_{\beta, \mu}$, defined on $L^\infty(\Omega)$,
has an extension to $L^2(\Omega)$. This allows us to establish irreducibility of $T_{\beta, \mu}^C$ via Propositions \ref{p.irreducible}
and \ref{p.zug} in the Markovian case, provided $\Omega$ is connected. On the other hand, we can use form methods to show that our assumptions to infer positivity resp.\ sub-Markovianity are close to optimal.

We consider the form $\mathfrak{a}_{\beta, h} : H^1(\Omega)\times H^1(\Omega) \to \CR$, given by
\[
\mathfrak{a}_{\beta, h}[u,v] \coloneqq \mathfrak{a}_\beta[u,v] - \int_{\partial\Omega} \int_{\Omega} h(z,x) u(x)\,\dx x \, v(z)\,\dx \sigma (z).
\]
Then the form $\mathfrak{a}_{\beta, h}$ is elliptic and continuous. Denote by $\cA_{\beta, h}^2$ the associated operator on $L^2(\Omega)$. Then $-\cA_{\beta, h}^2$ generates a holomorphic, strongly continuous semigroup $T_{\beta, h}^2$ on 
$L^2(\Omega)$. It is easy to see that if in addition
\begin{equation}\label{eq.additional}
\int_{\partial\Omega}\Big(\int_{\Omega} |h(z,x)|\dx x\Big)^p\,\dx \sigma,
\end{equation}
for some $p>d-1$ with $p\geq 2$, then the measures $\mu (z) = h(z, x)\dx x$ satisfy Hypothesis \ref{h.mu} whence we obtain a semigroup
$T_{\beta, \mu}$ on $L^\infty(\Omega)$ with generator $A_{\beta, \mu}$. Using the definition of the co-normal derivative one sees that 
the part of $-\cA_{\beta, h}^2$ in $L^\infty (\Omega)$ is precisely the operator $A_{\beta, \mu}$. It follows that
$T_{\beta, h}^2$ leaves the space $L^\infty(\Omega)$ invariant and the restriction of that semigroup to $L^\infty(\Omega)$ is 
$T_{\beta, \mu}$.

\begin{prop}\label{p.optimal}
With the notation above, we have:
\begin{enumerate}
[(a)]
\item The semigroup $T^2_{\beta, h}$ is positive if and only if $h \geq 0$ almost everywhere.
\item Assume that $b_j \in W^{1,\infty}(\Omega)$ for $j=1, \ldots, d$. Then  $T^2_{\beta,h}$ is sub-Markovian if and only if 
\eqref{eq.mubetasM1} holds, $h\geq 0$ almost everywhere and $0\leq \int_{\Omega}h(z,x)\dx x \leq \beta (z) + \sum_{j=1}^d \tr b_j(z) \nu_j (z)$ for almost every  $z\in \partial\Omega$.
\end{enumerate}
\end{prop}

\begin{proof}
(a)  By the first Beurling--Deny criterion \cite[Corollary 2.6]{mvv05} $T^2_{\beta,\mu}$ is positive if and only if $\mathfrak{a}_{\beta, \mu}[u^+, u^-]\leq 0$ for all $u\in H^1(\Omega)$. If $h\geq 0$ almost everywhere this is clearly fulfilled.
	
	Conversely assume that $T_{\beta,\mu}^2(t) \geq 0$ for all $t>0$. Then
\[
\int_{\partial\Omega}\int_{\Omega} h(z,x) u^+(x)\,\dx x \, u^- (z)\,\dx \sigma (z) = - \mathfrak{a}_{\beta, h}[u^+, u^-]\geq 0
\]
for all $u\in H^1(\Omega)$. Now let functions $0\leq v \in \mathscr{D}(\Omega)$ and $0\leq \varphi \in C(\partial\Omega)$ be given.
We find a sequence $w_n \in \mathscr{D}(\CR^d)$ with $0\leq w_n \leq \|\varphi\|_\infty$ such that $\supp w_n \cap \supp v = \emptyset$ and $w_n (z) \to \varphi (z)$ for all $z \in \partial\Omega$. Inserting $u= v- w_n$ in the above inequality and using dominated convergence, we obtain that
\[
\int_{\partial\Omega}\int_{\Omega} h(z,x) v(x)\,\dx x \, \varphi (z)\,\dx\sigma (z) \geq 0
\]
As $0\leq \varphi \in C(\partial\Omega)$ was arbitrary, we conclude that
\[
\int_{\Omega} h(z,x) v(x)\dx x \geq 0
\]
for almost all $z \in \partial\Omega$. As $0\leq v\in \mathscr{D}(\Omega)$ was arbitrary, it follows that for almost all $z\in \partial\Omega$ we have $h(z,x) =0$ for almost all $x \in \Omega$. Now Fubini's theorem implies that $h\geq 0$ with respect to the product measure, proving the necessity of the condition.\medskip

(b) The sufficiency of the inequality above was already established in Proposition \ref{p.evenmore}, so we only need to prove its necessity. If the semigroup is sub-Markovian, it is positive and thus $h\geq 0$ almost everywhere by (a).
	
By the  Beurling--Deny--Ouhabaz criterion \cite[Corollary 2.8]{mvv05}, for $u \in H^1(\Omega)$ we have
	\begin{align*}
		0& \leq \mathfrak{a}_{\beta, h}[u\wedge 1,(u-1)^+]\\
		& =-\sum_j\int_{\Omega} (D_jb_j)(u-1)^+\dx x+\int_{\Omega}d_0(u-1)^+\dx x\\
		&\quad +\int_{\partial\Omega}\Big(\sum_j b_j\nu_j(u-1)^++\beta(z)-\int_{\Omega}(u\wedge 1)(x)h(z,x)\dx x\Big) (u-1)^+(z)\,\dx \sigma(z).
	\end{align*}
Now let $v \in H^1(\Omega)$ such that $v\geq 0$. Inserting $u= v+\one$ in the above inequality, the desired inequalities follow from Lemma \ref{l.positive}.
\end{proof}

\begin{rem}
We have already noted after Proposition \ref{p.evenmore} that Condition \eqref{eq.mubetasM1} is necessary for $T_{\beta, \mu}$ to be sub-Markovian.
\end{rem}

We now consider the case where the semigroup is Markovian. Then we can prove irreducibility via Proposition \ref{p.irreducible}
and deduce convergence of the semigroup to an equilibrium.

\begin{thm}\label{t.convergence}
Assume that $\Omega$ is connected, and that $h\geq 0$ almost everywhere satisfies Equation \eqref{eq.additional}.
Moreover, assume that $\sum_{j=1}^d D_jb_j = d_0$ almost everywhere on $\Omega$ and
\[
\sum_{j=1}^d b_j(z)\nu_j(z) + \beta (z) = \int_{\Omega} h(z,x)\, \dx x
\] 
almost everywhere on $\partial \Omega$. Then the semigroup $T_{\beta, \mu}^C$ on $C(\overline{\Omega})$ is irreducible and Markovian. Consequently, there exist $0\ll \varphi \in L^2(\Omega)$ such that $\int_{\Omega} \varphi (x) \dx x = 1$
and constants $\eps, M >0$ such that
\[
\|T_{\beta, \mu}^C(t) - \varphi\otimes \one\|_{\mathscr{L}(C(\overline{\Omega}))} \leq Me^{-\eps t}
\]
for all $t>0$.
\end{thm}

\section{Measures satisfying Hypothesis \ref{h.mu}}\label{s.6}

In this brief section we give some examples of maps $\mu$ for which Hypothesis \ref{h.mu} is satisfied.

\begin{example}
Assume that for every Borel set $A \subset \overline{\Omega}$
the complex-valued map $z\mapsto \mu (z)(A)$ is continuous. Then $\mu$ satisfies conditions (a), (b) and (c) in Hypothesis \ref{h.mu}.
\end{example}

\begin{proof}
It is obvious that (a) holds. As for (b), we note that by continuity and compactness of $\partial \Omega$ we have
$\sup_{z\in \partial\Omega} |\mu(z)(A)| < \infty$ for every $A \in \mathscr{B}(\overline{\Omega})$. 
Now \cite[Corollary 4.6.4]{bogachev} yields $\sup_{z\in\partial\Omega}\|\mu(z)\| < \infty$. To prove (c), pick a dense sequence $z_n$ in $\partial \Omega$. We set
\[
\tau \coloneqq \sum_{n\in \CN} \frac{1}{2^n} |\mu (z_n)|,
\]
where $|\mu (z)|$ denotes the total variation of $\mu (z)$. Then $\tau$ is a finite positive measure and we have
$\mu (z_n) \ll \tau$ for every $n\in \CN$. Let $A \in \mathscr{B}(\overline{\Omega})$ with $\tau (A) = 0$ be given. Consider the function $\varphi (z) \coloneqq \mu(z)(A)$. By the above $\varphi (z_n)=0$ for all $n\in \CN$. Moreover, $\varphi$ is continuous by assumption. Thus $\varphi \equiv 0$, proving that in fact $\mu (z) \ll \tau$ for all $z\in\partial\Omega$.
\end{proof}
   
Similarly, we can consider maps $\mu$ which only take countably many values.

\begin{example}
Assume that $\mu (z) = \sum_{n\in J} \one_{A_n}(z)\mu_j$ where $(A_n)_{j\in J} \subset \mathscr{B}(\partial \Omega)$
and $(\mu_j)_{n\in J} \subset \mathscr{M}(\overline{\Omega})$ and $J$ is a finite or countably infinite index set. Then
$\mu$ satisfies Hypothesis \ref{h.mu} provided $\sum_{n\in J} \sigma (A_n)|\mu_n|(\overline{\Omega})^p < \infty$
where $p$ is as in Hypothesis \ref{h.mu}(b).
\end{example}

\begin{proof}
Part (a) is obvious and (b) was assumed. Part (c) is fulfilled with $\tau = \sum_{n\in J} 2^{-n} |\mu_n|$.
\end{proof}

\appendix
\section{Irreducible semigroups}\label{appendix}

In this appendix we collect some known facts on positive, irreducible semigroups. In some cases we present some variations or adapt results to our special situation.

Let $E$ be a real Banach lattice. In our context $E$ will be $C(\overline{\Omega})$ or $L^q(\Omega)$. Let $T$ be a strongly continuous semigroup on $E$ which is positive, i.e.\ for $f\in E_+$ we have $T(t)f\in E_+$ for all $t\geq 0$. We denote the generator
of $T$ by $A$. The \emph{spectral bound} of $A$ is defined by
\[
s(A) \coloneqq \sup\{\Re\lambda : \lambda \in \sigma (A_{\CC})\}
\]
where $\sigma (A_{\CC})$ is the spectrum of the generator $A_\CC$ of the complexification of $T$. In what follows, we will not distinguish between an operator and its complexification. In particular, when we talk about the spectrum, resolvent, etc.\ of an operator, we always mean the spectrum/resolvent, etc.\ of its complexification.

By \cite[C-III Theorem 1.1]{schreibmaschine}, $s(A) \in \sigma (A)$ whenever $\sigma (A) \neq \emptyset$ . If $A$ has compact resolvent,
then $\sigma (A)$ consists of isolated points which are all eigenvalues.

\begin{thm}
\label{t.asympt}
Assume that $T(t)$ is compact for all $t>0$, that $s(A) = 0$ and that $T$ is bounded. Then there exist a positive projection $P\neq 0$ of finite rank,
$\eps>0$ and $M> 0$ such that
\[
\|T(t)-P\|_{\mathscr{L}(E)} \leq Me^{-\eps t}
\]
for all $t>0$.
\end{thm}

\begin{proof}
Since $T(t)$ is compact for all $t>0$, $T$ is immediately norm continuous and it follows from \cite[C-III Corollary 2.13]{schreibmaschine} that there is some $\delta >0$ such that $\Re\lambda \leq -2\delta <0$ for all $\lambda 
\in \sigma (A)\setminus \{0\}$. Denote by $P$ the spectral projection with respect to $0$, i.e.\
\[
P \coloneqq \frac{1}{2\pi i} \int_{|\lambda| = \delta} R(\lambda, A) \,\dx \lambda.
\]
As $T(t)$ is compact for all $t>0$, so is the resolvent and thus also $P$, whence it has finite rank. The restriction of $T$ to
the range of $P$ is a bounded semigroup on a finite dimensional vector space whose generator has spectrum $\{0\}$. It follows 
that the generator of the restriction is diagonalizable and is thus the zero operator.  
Consequently, $T(t)P =P$ for all $t>0$. The space $F = (I-P)E$ is invariant
under the  semigroup and the generator $A_F$ of the restriction has its spectrum in a strict left half plane. Since the semigroup is immediately norm continuous there exist $\eps>0$, $M> 0$ such that $\|T(t)|_F\|_{\mathscr{L}(F)} \leq Me^{-\eps t}$
and hence $\|T(t) - P\|_{\mathscr{L}(E)}\leq Me^{-\eps t}$ for all $t\geq 0$.
\end{proof}

Theorem \ref{t.asympt} implies in particular that there exists $u>0$, i.e.\ $u\geq 0$ and $u\neq 0$, such that $T(t) u = u$ for all $t \geq 0$. Thus the Krein--Rutman Theorem which asserts that the largest eigenvalue (i.e.\ $s(A)$) has a positive eigenfunction is incorporated in Theorem \ref{t.asympt}. 

We next want to investigate when $P$ has rank one and the positive eigenfunction is strictly positive. This will be done via the notion
of \emph{irreducibility}. A subspace $J$ of $E$ is called an \emph{ideal} if
\begin{enumerate}
[(i)]
\item $u\in J$ implies $|u|\in J$ and
\item if $u \in J$, then $0\leq v \leq u$ implies $v \in J$.
\end{enumerate}
A positive, strongly continuous semigroup $T$ on $E$ is called \emph{irreducible} if the only invariant closed ideals are $J = \{0\}$
and $J=E$.

If $J=C(\overline{\Omega})$ then $J\subset E$ is a closed ideal if and only if there exists a closed subset $K$ of $\overline{\Omega}$
such that
\[
J = \{ f\in C(\overline{\Omega}) : f|_K = 0\}.
\]
If $E = L^q(\Omega)$ $(1\leq q <\infty)$ then $J\subset E$ is a closed ideal if and only if there exists a measurable subset $K$ of $\Omega$ such that 
\[
J = \{ f\in L^q(\Omega) : f|_K = 0 \mbox{ a.e.}\}.
\]

We say that  $u\in E$ is a \emph{quasi interior point} and write  $u\gg 0$ if the principal ideal
\[
E_u \coloneqq \{ v\in E : \exists \, c>0 \mbox{ such that } |v|\leq cu\}
\]
is dense in $E$.

If $E= C(\overline{\Omega})$ then $u\gg 0$ if and only if there is $\delta >0$ such that $u(x)\geq \delta >0$ for all $x\in \overline{\Omega}$.
In this case $u$ is actually an inner point of the positive cone. If $E= L^p(\Omega)$ then $u\gg 0$ if and only if $u(x) >0$ for almost every $x$.

We call  $\varphi \in E^\prime$  a \emph{strictly positive functional}  if $\langle \varphi, f\rangle =0$ implies $f=0$ for all $f \in E_+$. 

If $E= C(\overline{\Omega})$, then $\varphi$ is strictly positive if and only if there exists a strictly positive Borel measure $\nu$, i.e.\
$\nu (O) >0$ for all non-empty open sets $O\subset \overline{\Omega}$, such that
\[
\langle \varphi, f\rangle = \int_{\overline{\Omega}} f(x)\, \dx \nu (x).
\]
If $E= L^q(\Omega)$ for $\varphi \in L^{q'}(\Omega) \simeq (L^q(\Omega))^\prime$ to be strictly positive is equivalent to
that $\varphi (x) > 0$ almost everywhere, i.e.\  $\varphi \gg 0$.\medskip

The importance of these concepts in the study of asymptotic behavior stems from the fact that positive fixed points of positive, irreducible semigroups are \emph{strictly positive}. More precisely, if $T$ is a positive, irreducible, strongly continuous semigroup
and $u>0$ is such that $T(t) u = u$ for all $t>0$, then $u\gg 0$ and if $0<\varphi \in E^\prime$ is such that
$T(t)^\prime \varphi = \varphi$ for all $t>0$ then $\varphi$ is strictly positive. 
Moreover, because of irreducibility, $s(A)$ cannot be a pole of order larger than 1, see \cite[C-III Proposition 3.5]{schreibmaschine}.
This implies that $T(t)P = P$ for all $t>0$ in the proof of Theorem \ref{t.asympt} even though the semigroup is not assumed to be bounded. We thus obtain the following result on asymptotic stability.

\begin{thm}
\label{t.asymp2}
Let $T$ be a positive, irreducible strongly continuous semigroup on $E$ with generator $A$. Assume that $T(t)$ is compact for $t>0$ and $s(A) = 0$. Then there exist $0\ll u \in \ker A$, a strictly positive $\varphi \in \ker A^\prime$, $\eps >0$, $M> 0$ such that
$\langle \varphi , u \rangle = 1$ and
\[
\|T(t) - \varphi\otimes u\|_{\mathscr{L}(E)} \leq M^{-\eps t}
\]
for all $t\geq 0$ where we have written $\varphi \otimes u$ for the projection defined by
\[
(\varphi\otimes u)(f) = \langle \varphi , f\rangle u,
\]
for all $f \in E$.
In particular
\[
\lim_{t\to \infty} T(t)f = \langle \varphi, f\rangle u,
\]
i.e.\ the orbits of the semigroup converge to an equilibrium.
\end{thm}

Theorems \ref{t.asympt} and \ref{t.asymp2} lie at the heart of the Perron--Frobenius theory. We refer to \cite{schreibmaschine} for more information.\medskip

We shall have occasion to use the strict monotonicity of the spectral bound.

\begin{thm}\label{t.monotone}
Let $S$ and $T$ be strongly continuous semigroups on $E$ with generators $B$ and $A$ respectively. Assume that
\begin{enumerate}
[(i)]
\item $0\leq S(t) \leq T(t)$ for all $t>0$;
\item $A$ has compact resolvent, and
\item $T$ is irreducible.
\end{enumerate}
If $A\neq B$, then $s(B) < s(A)$.
\end{thm}

\begin{proof}
This is a version of \cite[Theorem 1.3]{ab92}, see also \cite[Theorem 10.2.10]{A05} in connection with \cite[Theorems 10.6.3 and 10.6.1]{A05}.
\end{proof}

Next we describe ways to prove irreducibility. On $L^2(\Omega)$ this is very easy if the semigroup is associated with 
a form by virtue of the Beurling--Deny--Ouhabaz criterion for the invariance  of closed convex sets. In particular 
the following holds true (see \cite[Theorem 2.10]{ouh05}).

\begin{prop}\label{p.irreducible}
Let $V \subset H^1(\Omega)$ be a closed subspace containing $H^1_0(\Omega)$, where $\Omega \subset \CR^d$ is a connected, open set. Let $\mathfrak{a}: V\times V \to \CR$ be a continuous and elliptic form such that the associated semigroup $T$ is positive. Then
$T$ is irreducible.
\end{prop}

On $C(\overline{\Omega})$ irreducibility is a stronger notion than on $L^2(\Omega)$. However, the following result shows how irreducibility on $C(\overline{\Omega})$ can be deduced from irreducibility on $L^2(\Omega)$.

\begin{prop}\label{p.zug}
Let $\Omega \subset \CR^d$ be open and bounded and $T$ be a positive, irreducible, strongly continuous semigroup 
on $L^2(\Omega)$ whose generator $A$ has compact resolvent. Assume that $T$ leaves $C(\overline{\Omega})$ invariant and that the restriction $T^C$ of $T$ to $C(\overline{\Omega})$ is strongly continuous and suppose that its generator $A^C$ has 
compact resolvent. Assume that
$s(A) =0$. Then $T^C$ is irreducible if and only if there exists $u\in \ker A\cap C(\overline{\Omega})$ such that $u(x) \geq \delta >0$
for all $x\in \overline{\Omega}$.
\end{prop}

\begin{proof}
Assume that there exists $0\ll u \in C(\overline{\Omega})\cap \ker A$. Since $T$ is irreducible $0$ is a pole of order 1 and the residuum
$P$ is of the form
\[
Pf = \Big(\int_\Omega \varphi f\,\dx x\Big)\cdot u
\]
for some $0\ll \varphi \in L^2(\Omega)$, see \cite[C-III Proposition 3.5]{schreibmaschine}. Since $C(\overline{\Omega})$ is dense in $L^2(\Omega)$, it follows that the coefficients in the Laurent series expansion in $C(\overline{\Omega})$ around $0$ (see \cite[A-III, Equation (3.1)]{schreibmaschine}) are the restriction of those in $L^2(\Omega)$. 
Thus $0$ is also a pole of order 1 of the resolvent of $A^C$. The residuum
\[ P^C = \frac{1}{2\pi i} \int_{|\lambda|= \eps} R(\lambda, A^C)\,\dx \lambda \]
is the same, i.e.\ $P^C= P|_{C(\overline{\Omega})}$.
Now let $J = \{f \in C(\overline{\Omega}) : f|_K = 0\}$ be an invariant ideal. Then for $z \in K$, $f \in J$, $f\geq 0$ we have 
$(T(t)f)(z) =0$ for all $t>0$ and hence $(R(\lambda, A^C)f)(z)=0$ for all $\lambda >0$, since we suppose
that $s(A)=0$ and know that $s(A)$ is the abscissis of the Laplace transform of the semigroup \cite[Theorem 5.3.1]{abhn}. Thus
\[
\int_{\Omega} f(x)\varphi (x)\dx x\cdot u(z) = \lim_{\lambda\downarrow 0} (\lambda R(\lambda, A^C)f)(z) = 0.
\]
Since $\varphi \gg 0$ in $L^2(\Omega)$ this implies $f= 0$. Consequently $J = \{0\}$. This proves the sufficiency.

To show the necessity, recall that $0$ is also a pole of $R(\lambda, A^C)$. It follows that $s(A^C) = 0$. By Theorem \ref{t.asymp2}, 
there exists $0\ll u \in \ker (A^C) \subset \ker (A)$.
\end{proof}

\section*{Acknowledgement}

The authors are grateful to Jochen Gl\"uck for discussions  on a train from Ulm to Munich which led to Proposition \ref{p.zug}.


\begin{thebibliography}{10}

\bibitem{A15}
{\sc M.~S. Agranovich}, {\em Sobolev spaces, their generalizations and elliptic
  problems in smooth and {L}ipschitz domains}, Springer Monographs in
  Mathematics, Springer, Cham, 2015.
\newblock Revised translation of the 2013 Russian original.

\bibitem{a94}
{\sc W.~Arendt}, {\em Gaussian estimates and interpolation of the spectrum in
  {$L^p$}}, Differential Integral Equations, 7 (1994), pp.~1153--1168.

\bibitem{A04}
\leavevmode\vrule height 2pt depth -1.6pt width 23pt, {\em Semigroups and
  evolution equations: Functional calculus, regularity and kernel estimates},
  vol.~1 of Handbook of Differential Equations: Evolutionary Equations,
  North-Holland, 2002, pp.~1 -- 85.

\bibitem{A05}
\leavevmode\vrule height 2pt depth -1.6pt width 23pt, {\em Heat kernels.
  lecture notes of the 9th internet seminar}.
\newblock
  {https://www.uni-ulm.de/fileadmin/website\_uni\_ulm/mawi.inst.020/arendt/downloads/internetseminar.pdf},
  2005.

\bibitem{ab92}
{\sc W.~Arendt and C.~J.~K. Batty}, {\em Domination and ergodicity for positive
  semigroups}, Proc.\ Amer.\ Math.\ Soc., 114 (1992), pp.~743--747.

\bibitem{abhn}
{\sc W.~Arendt, C.~J.~K. Batty, M.~Hieber, and F.~Neubrander}, {\em
  Vector-valued {L}aplace transforms and {C}auchy problems}, vol.~96 of
  Monographs in Mathematics, Birkh\"auser/Springer Basel AG, Basel, second~ed.,
  2011.

\bibitem{schreibmaschine}
{\sc W.~Arendt, A.~Grabosch, G.~Greiner, U.~Groh, H.~P. Lotz, U.~Moustakas,
  R.~Nagel, F.~Neubrander, and U.~Schlotterbeck}, {\em One-parameter semigroups
  of positive operators}, vol.~1184 of Lecture Notes in Mathematics,
  Springer-Verlag, Berlin, 1986.

\bibitem{akk16}
{\sc W.~Arendt, S.~Kunkel, and M.~Kunze}, {\em Diffusion with nonlocal boundary
  conditions}, J. Funct. Anal., 270 (2016), pp.~2483--2507.

\bibitem{am}
{\sc W.~Arendt and R.~Mazzeo}, {\em Spectral properties of the
  dirichlet-to-neumann operator}, Ulmer Seminare, 12 (2007), pp.~23--37.

\bibitem{AtE97}
{\sc W.~Arendt and A.~F.~M. ter Elst}, {\em Gaussian estimates for second order
  elliptic operators with boundary conditions}, J. Operator Theory, 38 (1997),
  pp.~87--130.

\bibitem{AtEff}
\leavevmode\vrule height 2pt depth -1.6pt width 23pt, {\em The
  {D}irichlet-to-{N}eumann operator on \protect{$C(\partial\Omega)$}}.
\newblock preprint, 2017.

\bibitem{bogachev}
{\sc V.~I. Bogachev}, {\em Measure theory. {V}ol. {I}, {II}}, Springer-Verlag,
  Berlin, 2007.

\bibitem{bkl01}
{\sc D.~M. Bo{\v{s}}kovi\'c, M.~Krsti\'c, and W.~Liu}, {\em Boundary control of
  an unstable heat equation via measurement of domain-averaged temperature},
  IEEE Trans. Automat. Control, 46 (2001), pp.~2022--2028.

\bibitem{Dan00}
{\sc D.~Daners}, {\em Heat kernel estimates for operators with boundary
  conditions}, Math. Nachr., 217 (2000), pp.~13--41.

\bibitem{Dan09}
\leavevmode\vrule height 2pt depth -1.6pt width 23pt, {\em Inverse positivity
  for general {R}obin problems on {L}ipschitz domains}, Arch. Math. (Basel), 92
  (2009), pp.~57--69.

\bibitem{dp86}
{\sc B.~de~Pagter}, {\em Irreducible compact operators}, Math. Z., 192 (1986),
  pp.~149--153.

\bibitem{d16}
{\sc D.~Dier}, {\em Non-autonomous forms and invariance}.
\newblock preprint. available at arXiv:1609.03857, 2016.

\bibitem{en}
{\sc K.-J. Engel and R.~Nagel}, {\em One-parameter semigroups for linear
  evolution equations}, vol.~194 of Graduate Texts in Mathematics,
  Springer-Verlag, New York, 2000.
\newblock With contributions by S. Brendle, M. Campiti, T. Hahn, G. Metafune,
  G. Nickel, D. Pallara, C. Perazzoli, A. Rhandi, S. Romanelli and R.
  Schnaubelt.

\bibitem{feller-semigroup}
{\sc W.~Feller}, {\em The parabolic differential equations and the associated
  semi-groups of transformations}, Ann. of Math. (2), 55 (1952), pp.~468--519.

\bibitem{feller-diffusion}
\leavevmode\vrule height 2pt depth -1.6pt width 23pt, {\em Diffusion processes
  in one dimension}, Trans. Amer. Math. Soc., 77 (1954), pp.~1--31.

\bibitem{feller57}
\leavevmode\vrule height 2pt depth -1.6pt width 23pt, {\em Generalized second
  order differential operators and their lateral conditions}, Illinois J.
  Math., 1 (1957), pp.~459--504.

\bibitem{galaskub}
{\sc E.~I. Galakhov and A.~L. Skubachevski\u\i}, {\em On {F}eller semigroups
  generated by elliptic operators with integro-differential boundary
  conditions}, J. Differential Equations, 176 (2001), pp.~315--355.

\bibitem{greiner}
{\sc G.~Greiner}, {\em Perturbing the boundary conditions of a generator},
  Houston J. Math., 13 (1987), pp.~213--229.

\bibitem{gjs08}
{\sc P.~L. Gurevich, W.~J{\"a}ger, and A.~L. Skubachevski\u\i}, {\em On the
  existence of periodic solutions of some nonlinear problems of thermal
  control}, Dokl. Akad. Nauk, 418 (2008), pp.~151--154.

\bibitem{k11}
{\sc M.~Kunze}, {\em A {P}ettis-type integral and applications to transition
  semigroups}, Czechoslovak Math. J., 61 (2011), pp.~437--459.

\bibitem{k13a}
\leavevmode\vrule height 2pt depth -1.6pt width 23pt, {\em Perturbation of
  strong {F}eller semigroups and well-posedness of semilinear stochastic
  equations on {B}anach spaces}, Stochastics, 85 (2013), pp.~960--986.

\bibitem{lotz}
{\sc H.~P. Lotz}, {\em Uniform convergence of operators on {$L^\infty$} and
  similar spaces}, Math. Z., 190 (1985), pp.~207--220.

\bibitem{lunardi}
{\sc A.~Lunardi}, {\em Analytic semigroups and optimal regularity in parabolic
  problems}, Progress in Nonlinear Differential Equations and their
  Applications, 16, Birkh\"auser Verlag, Basel, 1995.

\bibitem{mvv05}
{\sc A.~Manavi, H.~Vogt, and J.~Voigt}, {\em Domination of semigroups
  associated with sectorial forms}, J. Operator Theory, 54 (2005), pp.~9--25.

\bibitem{N11}
{\sc R.~Nittka}, {\em Regularity of solutions of linear second order elliptic
  and parabolic boundary value problems on {L}ipschitz domains}, J.
  Differential Equations, 251 (2011), pp.~860--880.

\bibitem{ouh05}
{\sc E.~M. Ouhabaz}, {\em Analysis of heat equations on domains}, vol.~31 of
  London Mathematical Society Monographs Series, Princeton University Press,
  Princeton, NJ, 2005.

\bibitem{su65}
{\sc K.~Sato and T.~Ueno}, {\em Multi-dimensional diffusion and the {M}arkov
  process on the boundary}, J. Math. Kyoto Univ., 4 (1964/1965), pp.~529--605.

\bibitem{sk89}
{\sc A.~L. Skubachevski\u\i}, {\em Some problems for multidimensional diffusion
  processes}, Dokl. Akad. Nauk SSSR, 307 (1989), pp.~287--291.
\newblock translation in Soviet Math.\ Dokl.\ 40 (1990), no.\ 1, 75–-79.

\bibitem{sk95}
\leavevmode\vrule height 2pt depth -1.6pt width 23pt, {\em Nonlocal elliptic
  problems and multidimensional diffusion processes}, Russian J. Math. Phys., 3
  (1995), pp.~327--360.

\bibitem{taira}
{\sc K.~Taira}, {\em Semigroups, boundary value problems and {M}arkov
  processes}, Springer Monographs in Mathematics, Springer, Heidelberg,
  second~ed., 2014.

\bibitem{taira_analytic}
\leavevmode\vrule height 2pt depth -1.6pt width 23pt, {\em Analytic semigroups
  and semilinear initial boundary value problems}, vol.~434 of London
  Mathematical Society Lecture Note Series, Cambridge University Press,
  Cambridge, second~ed., 2016.

\bibitem{wenzell}
{\sc A.~D. Ventcel$'$}, {\em On boundary conditions for multi-dimensional
  diffusion processes}, Theor. Probability Appl., 4 (1959), pp.~164--177.

\end{thebibliography}
\end{document}